\documentclass[12pt,a4paper]{amsart}
\usepackage{amsfonts,amssymb,amsmath,amscd,amsthm,enumerate,hyperref}

\usepackage{comment}

\usepackage{graphicx}

\usepackage{color}

\usepackage[margin=1.1in]{geometry}
\newtheorem{theorem}{Theorem} 
\newtheorem*{theorem*}{Theorem}
\newtheorem{lemma}[theorem]{Lemma}
\newtheorem{definition}[theorem]{Definition}
\newtheorem{proposition}[theorem]{Proposition}
\newtheorem{conjecture}[theorem]{Conjecture}
\newtheorem{corollary}[theorem]{Corollary}

\theoremstyle{remark}

\newcommand{\D}{\mathcal{D}}

\newcommand{\ii}{{\mathbf i}}

\newcommand{\eps}{\varepsilon}

\numberwithin{equation}{section} \numberwithin{theorem}{section}

\newcommand{\hG}{\widehat G}

\newcommand{\CG}{\mathcal S}

\newcommand{\mes}{\mathbf m}

\renewcommand{\H}{H}

\title[Lie groups and quantized free convolution]
{Representations of classical Lie groups and quantized free convolution}

\author{Alexey Bufetov}

\address[Alexey Bufetov]{Department of Mathematics, Higher School of Economics, Moscow, Russia, and Institute for
Information Transmission Problems of Russian Academy of Sciences, Moscow, Russia. E-mail: alexey.bufetov@gmail.com}

\author{Vadim Gorin}

\address[Vadim Gorin]{Department of Mathematics, Massachusetts Institute of Technology, Cambridge, MA, USA, and
 Institute for Information Transmission Problems of Russian Academy of Sciences, Moscow, Russia. E-mail: vadicgor@gmail.com}

\begin{document}

\maketitle

\begin{abstract}
We study the decompositions into irreducible components of tensor products and restrictions of
irreducible representations for all series of classical Lie groups as the rank of the group goes
to infinity. We prove the Law of Large Numbers for the random counting measures describing the
decomposition. This leads to two operations on measures which are deformations of the notions of
the free convolution and the free projection. We further prove that if one replaces counting
measures with others coming from the work of Perelomov and Popov on the higher order Casimir
operators for classical groups, then the operations on the measures turn into the free convolution
and projection themselves. We also explain the relation between our results and limit shape
theorems for uniformly random lozenge tilings with and without axial symmetry.
\end{abstract}

\section{Introduction}

\subsection{Summary}

We start by stating one of the results of the present article. Let $U(N)$ denote the (compact Lie)
group of all $N\times N$ complex unitary matrices. Due to E.~Cartan and H.~Weyl (see e.g.\
\cite{Weyl-book}) all irreducible representations of $U(N)$ are parameterized by their highest
weights, which are \emph{signatures}
--- $N$--tuples of integers $\lambda_1\ge\lambda_2\ge\dots\ge\lambda_N$. We denote by
$\widehat U(N)$ the set of all signatures and by $\pi^\lambda$ the representation corresponding to
the signature $\lambda$.

One way to encode a signature $\lambda$ is through the \emph{counting measure} $m[\lambda]$
corresponding to it via
\begin{equation}
\label{eq_def_uniform_intro}
 m[\lambda]=\frac{1}{N}\sum_{i=1}^N \delta\left(\frac{\lambda_i+N-i}N\right).
\end{equation}
Clearly, $m[\lambda]$ is a discrete probability measure on $\mathbb R$. This procedure is
illustrated in Figure \ref{Figure_mes}.

\begin{figure}[h]
\begin{center}
 {\scalebox{0.9}{\includegraphics{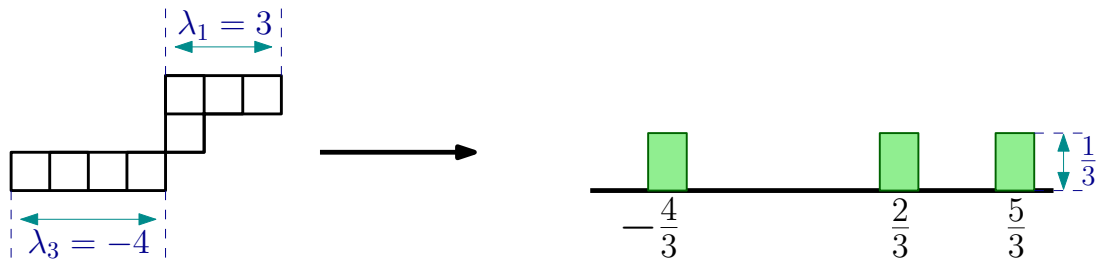}}}
 \caption{Counting measure corresponding to  signature $3\ge1\ge -4$
 \label{Figure_mes}}
\end{center}
\end{figure}

Given a finite-dimensional representation $\pi$ of $U(N)$ we can decompose it into irreducible
components:
\begin{equation}
\label{eq_decomp_irrep}
 \pi = \bigoplus_{\lambda\in \widehat U(N)} c_\lambda \pi^{\lambda},
\end{equation}
where non-negative integers $c_\lambda$ are multiplicities. The decomposition
\eqref{eq_decomp_irrep} can be identified with a probability measure $\rho^\pi$ on $\widehat U(N)$
such that
\begin{equation}
\label{eq_rep_measure}
 \rho^\pi(\lambda)= \frac{c_\lambda \dim(\pi^{\lambda})}{\dim(\pi)},
\end{equation}
where $\dim(\pi)$ is the dimension of $\pi$. In other words, $\rho^\pi$ weights a signature
according to the relative size of its isotypical component in $\pi$. The pushforward of $\rho^\pi$
with respect to the map $\lambda\to m[\lambda]$ is a \emph{random} probability measure on $\mathbb
R$ that we denote $m[\rho^\pi]$.

One of the main results of the present article is the Law of Large Numbers for $m[\rho^\pi]$, when
$\pi$ is a \emph{tensor product} of two irreducible representations of $U(N)$ and $N$ is large.
Since the decomposition of a tensor product into irreducible components is given by the classical
\emph{Littlewood--Richardson rule} (cf.\ \cite{LR}, \cite[Chapter I, Section 9]{M}), one can say
that we study the asymptotics of the Littlewood--Richardson coefficients.

\begin{theorem}
\label{Theorem_main_intro}
 Suppose that $\lambda^1(N),\lambda^2(N)\in\widehat U(N)$, $N=1,2,\dots$, are $2$ sequences of
 signatures which satisfy a technical assumption of Definition \ref{definition_regularity} and
 such that
 $$
  \lim_{N\to\infty} m[\lambda^i(N)]=\mes^i,\text{ (weak convergence), } i=1,2.
 $$
 Let $\pi(N)=\pi^{\lambda^1(N)}\otimes \pi^{\lambda^2(N)}$. Then as $N\to\infty$ random measures $m[\rho^{\pi(N)}]$ converge in the sense of
moments, in probability to a deterministic measure which we denote $\mes^1 \otimes \mes^2$.
\end{theorem}

Here is a summary of the results of the article:
\begin{enumerate}
\item We prove Theorem \ref{Theorem_main_intro} and its analogues for \emph{all series} of classical Lie groups,
i.e.\ for unitary groups, symplectic groups, and orthogonal groups in odd and even dimensions, see
Theorem \ref{Theorem_main}.

\item We investigate the operation on measures $(\mes^1,\mes^2)\mapsto\mes^1 \otimes \mes^2$ and show that it can be
described as a \emph{deformation} of the well-known notion of the free convolution (see
\cite{VDN},\cite{NS} for the overview of the free probability theory). We further call $\mes^1
\otimes \mes^2$ the \emph{quantized free convolution} of measures $\mes^1$ and $\mes^2$. The
formula for its computation can be found in Theorem \ref{Theorem_linearization} below. In Section
\ref{section_MK} we explain how this operation is related to the conventional (additive) free
convolution.

\item We show in Theorem \ref{Theorem_main_PP} that if one replaces the counting measure \eqref{eq_def_uniform_intro} by another
remarkable probability measure, then an analogue of Theorem \ref{Theorem_main_intro} would involve
\emph{precisely} the  (additive) free convolution. The definition of the probability measure that
we use, is inspired by the work of Perelomov and Popov \cite{PP} on Casimir elements in the
enveloping algebras of classical Lie groups.

\item We study another analogue of Theorem \ref{Theorem_main_intro} for classical Lie groups in
which tensor products are replaced by the restrictions to smaller subgroups. The result turns out
to be related to the free projection (i.e.\ free compression with a free projector) from the free
probability theory and its deformation, see Theorems \ref{Theorem_main}, \ref{Theorem_main_PP},
\ref{Theorem_linearization}.

\end{enumerate}

In the rest of the Introduction we give a historic overview and explain our main results. Careful
formulations of all our theorems can be found in Section \ref{Section_Setup}, and proofs are in
Sections \ref{Section_characters}--\ref{Section_appendix}.

In Section \ref{Section_discussion} we explain various connections around our results: In section
\ref{Section_symmetric} we link our theorems to the results of Biane \cite{Biane_S} on the
asymptotics of the decompositions of representations of the symmetric groups $S(N)$ as
$N\to\infty$. In Section \ref{Section_tilings} we present interpretations of our results in terms
of \emph{random lozenge tilings}. These interpretations make the combinatorial similarities between
the cases of unitary, symplectic and orthogonal groups especially transparent. Finally, in Section
\ref{Section_U_infty} we explain the connection to the limit shape theorems for the characters of
$U(\infty)$.

\subsection{Historic overview}
\label{Section_History} The symmetric group $S(N)$ of permutations of $N$ elements and the unitary
group $U(N)$ of $N\times N$ complex unitary matrices are the model examples of a (noncommutative)
finite and a compact group, respectively. Both series of groups depend on an integer parameter $N$
and the study of the behavior of such groups and their representations as $N\to\infty$ is now known
as the \emph{asymptotic representation theory.}

The study of the asymptotic questions was initiated by Thoma \cite{Thoma-S} and Voiculescu
\cite{Vo_U} who were interested in the classification of \emph{characters} (which are
positive--definite conjugation--invariant continuous functions on a group) and von Neumann
\emph{finite factor representations} of type $II_1$ for the infinite symmetric group
$S(\infty)=\bigcup_{N=1}^\infty S(N)$ and the infinite--dimensional unitary group
$U(\infty)=\bigcup_{N=1}^{\infty} U(N)$. As opposed to the finite $N$ case, the characters of
$S(\infty)$ and $U(\infty)$ depend on infinitely many continuous parameters --- this is one of the
manifestations of the fact that these groups are ``big''. It was discovered by Vershik--Kerov
\cite{VK_U} and Boyer \cite{Bo} that in a hidden form the classification of characters is implied
by the work of Aissen, Edrei,
 Schoenberg, and Whitney \cite{AESW}, \cite{Ed} on totally positive Toeplitz matrices.

Later Vershik and Kerov \cite{VK_S}, \cite{VK_U} gave an alternative approach to the characters of
$S(\infty)$ and $U(\infty)$: they showed that each such character can be \emph{approximated} by a
sequence of normalized conventional characters of irreducible representations of $S(N)$ and $U(N)$,
respectively, as $N\to\infty$. This approach was further developed and generalized in
\cite{Boyer2}, \cite{KOO}, \cite{OkOlsh}, \cite{OkOlsh_BC}, \cite{BO-newA}, \cite{Petrov-boundary},
\cite{GP}, leading, in particular, to classification theorems for characters of infinite
dimensional orthogonal and symplectic groups ($SO(\infty)$ and $Sp(\infty)$) and spherical
functions of certain infinite--dimensional Gelfand pairs.

The sequences of irreducible representations arising in the Vershik--Kerov approximation theory
have a very special form that we now describe. Recall that irreducible representations of $U(N)$
are parameterized by \emph{signatures} which are $N$--tuples of integers
$\lambda_1\ge\lambda_2\ge\dots\ge\lambda_N$. One necessary condition for a sequence of signatures
$\lambda(N)$ to be an approximating sequence for a character of $U(\infty)$ is that as
$N\to\infty$, for each $i$, $\frac{1}{N}\lambda_i(N)$ and also $\frac{1}{N}\sum_{i=1}^N
|\lambda_i(N)|$ converge to finite limits, see \cite{VK_U}, \cite{OkOlsh} for the details. In other
words, all coordinates of $\lambda$ should grow linearly in $N$ and the sum of all $N$ coordinates
should also grow linearly. Clearly, these are very special \emph{``thin''} signatures. For
symmetric, orthogonal and symplectic groups the situation is very similar.

\smallskip

The above discussion naturally leads to the question: Is there any concise asymptotic
representation theory which would describe the limit behavior or irreducible representations whose
signatures are not ``thin''?

\smallskip

This question was first addressed by Biane for unitary groups in \cite{Biane_U} and for symmetric
groups in \cite{Biane_S}. For $U(N)$ he considered  \emph{``ultra--thick''} signatures, i.e.\ those
whose coordinates grow superlinearly as $N\to\infty$.  Let $\eps(N)$ be a sequence of positive
reals such that $\lim_{N\to\infty} \eps(N) N^a=0$ for all $a=1,2,\dots$. The idea of Biane was to
identify signature $\lambda=(\lambda_1\ge \lambda_2\ge\dots\ge\lambda_N)$ with discrete probability
measure on $\mathbb R$
\begin{equation}
\label{eq_Biane_measure}
 m_{Biane}[\lambda]=\frac{1}{N}\sum_{i=1}^N \delta(\eps(N) \lambda_i).
\end{equation}
Biane studied the asymptotic behavior of sequences of irreducible representations of unitary
groups, such that the corresponding measures $m_{Biane}[\lambda(N)]$ weakly converge. In
particular, he showed that if one considers two such sequences $\lambda^1(N)$, $\lambda^2(N)$
$$
 \lim_{N\to\infty} m_{Biane}[\lambda^1(N)]=\mes^1,\quad  \lim_{N\to\infty} m_{Biane}[\lambda^2(N)]=\mes^2,
$$
then the measure corresponding to the typical irreducible component in tensor product
$\pi^{\lambda^1(N)}\otimes\pi^{\lambda^2(N)}$ (in the same sense as in Theorem
\ref{Theorem_main_intro}) converges to a deterministic measure $\mes^1\boxplus\mes^2$. Moreover,
the operation $(\mes^1,\mes^2)\mapsto(\mes^1\boxplus\mes^2)$ is precisely the \emph{free
convolution} from the free probability theory. Recently Collins and Sniady \cite{CS} showed that
the restrictions on $\eps(N)$ can be significantly weakened and the results of Biane \cite{Biane_U}
still hold when merely $\lim_{N\to\infty} \eps(N)\cdot N =0$.

On the contrast, in our Theorem \ref{Theorem_main_intro}, the scaling of the coordinates
$\lambda_i$ is linear. Moreover, the result also becomes conceptually different, the operation
$(\mes^1,\mes^2)\mapsto\mes^1 \otimes \mes^2$ is \emph{not} the free convolution.

One could ask whether there are any \emph{a priori} reasons to distinguish the scaling regime of
Theorem \ref{Theorem_main_intro} (when $\lambda_i(N)$ grows linearly in $N$) from those considered
earlier by Biane \cite{Biane_U}, and Collins--Sniady \cite{CS}. We give several such reasons
below.

In Section \ref{Section_superlinear} we put our results as well as those of \cite{Biane_U},
\cite{Colins-Sniady} into a context of random matrix theory. This shows the key difference between
Theorem \ref{Theorem_main_intro} (and its generalizations) and theorems of \cite{Biane_U},
\cite{CS}, and, to a certain extent, explains why one should expect the direct relation to the free
probability for the latter, but should not for the former.

In Section \ref{Section_tilings} we relate our results to the study of uniformly random
 \emph{lozenge tilings} of polygonal domains. For these polygons the linear dependence of
 $\lambda_i(N)$ on $N$ transforms into the natural assumption of boundedness of the ratios of side lengths as
 $N\to\infty$, which links our theorems to the \emph{limit shape theorems} for tilings of planar domains by
 Cohn--Kenyon--Propp \cite{CKP}, Kenyon--Okounkov--Sheffield \cite{KOS}, Kenyon--Okounkov
 \cite{KO-Burgers}. On the other hand, when $\lambda_i(N)$ grow superlinerly, the domains become degenerate and
 this connection to a large extent disappears.

Further, let us note that in many problems of the asymptotic representation theory there exists a
certain symmetry between horizontal coordinate (which typically corresponds to the \emph{rows} of
the involved Young diagrams) and vertical coordinate (similarly corresponding to \emph{columns}),
cf.\ \cite{VK_S}, \cite{VK_U}, \cite{KOO}, \cite{OkOlsh}, \cite{OkOlsh_BC}. Our limit regime
preserves this symmetry: Indeed, in the graphical illustration of the left panel of Figure
\ref{Figure_mes} we scale both vertical and horizontal coordinates by $N$ (and see a continuous
profile in the limit). On the other hand in the regime of \cite{Biane_U}, \cite{CS} the vertical
and horizontal scalings differ from each other.

Finally, our limit regime is intimately related
 to the character theory for the infinite--dimensional unitary group $U(\infty)$ and recent limit shape theorems of \cite{BBO},
  we elaborate on this connection
  in more detail in Section \ref{Section_U_infty}.

\smallskip

From another direction, in \cite{Biane_S} Biane studied the asymptotics of the decompositions for
restrictions and products of irreducible representations of symmetric groups $S(n)$ as
$n\to\infty$. Recall that irreducible representations of $S(n)$ are parameterized by Young diagrams
with $n$ boxes; the paper \cite{Biane_S} concentrates on ``\emph{balanced}'' Young diagrams, whose
rows and columns grow as $c\sqrt{n}$, which are again opposed to ``thin'' Young diagrams appearing
in the Vershik--Kerov theory. Biane proves the limit shape theorems similar to our Theorem
\ref{Theorem_main_intro} for the operations on irreducible representations on $S(n)$ and also finds
a connection to the free probability.

It is well-known that in many aspects the asymptotic representation theory for $S(n)$ and for
classical Lie groups are parallel. Moreover, recently Borodin and Olshanski \cite{BO-US} showed
how the asymptotic representation theory corresponding to the infinite-dimensional unitary group
$U(\infty)$ can be degenerated into the one for $S(\infty)$. From this point of view, our theorems
can be viewed as the \emph{lifting} of the results of Biane \cite{Biane_S} from the level of
symmetric groups $S(n)$ up to the level of classical Lie groups $U(N)$, $SO(N)$, $Sp(2N)$. In
particular, an important role in \cite{Biane_S} is played by the so-called Kerov transition
measure of the Young diagram and we will show that its exact analogue for classical groups is
given by the Perelemov--Popov measures. More details on the limit transition to $S(n)$ are
provided in Section \ref{Section_symmetric}.

\subsection{Free convolution and semiclassical limit}
\label{Section_superlinear} The aim of this section is to put our results and those of
\cite{Biane_U}, \cite{CS} in the context of random matrix theory and free probability.

The notion of the free convolution was originally defined by Voiculescu \cite{Vo} in the setting of
operator algebras. However, one can explain this notion using only certain generating functions.
First, recall that the usual notion of the convolution of measures is nicely related to
characteristic functions. Namely, if $\phi_1(z)$ is the characteristic function of a probability
measure $\mes^1$, i.e.\ $\phi_1(z)=\int_{\mathbb R} \exp(\ii x\cdot z) \mes^1(dx)$, and $\phi_2(z)$
is the characteristic function of a probability measure $\mes^2$, then the characteristic function
of the convolution of $\mes^1$ and $\mes^2$ is the product $\phi_1(z)\phi_2(z)$. Equivalently, in
the probabilistic language we say that the characteristic function of the sum of independent random
variables is the product of the characteristic functions. Another way to state the same thing is
that when the measures are convoluted, \emph{logarithms} of their characteristic functions are
added.

The free convolution can be defined in the same way with logarithm of the characteristic function
replaced by the so--called Voiculescu $R$--transform of the measure (we recall the definition of
this notion in \eqref{eq_Voiculescu_R})

\begin{definition}[Voiculescu, \cite{Vo}, \cite{Vo2}]
\label{def_free}
 The free convolution is a unique operation on probability measures
 $(\mes^1,\mes^2)\mapsto\mes^1\boxplus\mes^2$, which agrees with the addition of
 $R$--transforms:
$$
 R_{\mes^1}(z)+ R_{\mes^2}(z)= R_{\mes^1\boxplus\mes^2}(z).
$$
\end{definition}

Quantized free convolution $(\mes^1,\mes^2)\mapsto\mes^1 \otimes \mes^2$ appearing in Theorem
\ref{Theorem_main_intro} can be also defined along these lines. Theorem
\ref{Theorem_linearization} below claims that one gets the quantized free convolution by replacing
all the instances of $R$--function $R_{\mes}(z)$ in Definition \ref{def_free} with
\begin{equation}
\label{eq_R_intro}
 R^{quant}_\mes(z)=R_\mes(z)+\frac{1}{z}-\frac{1}{1-e^{-z}}=R_\mes(z)-R_{u[0,1]}(z),
\end{equation}
where $u[0,1]$ is the uniform measure on the interval $[0,1]$. Note, however, that in Theorem
\ref{Theorem_main_intro} only the measures which have bounded by $1$ density with respect to the
Lebesgue measure appear and the definition of the quantized free convolution is restricted only to
this class of measures.

\bigskip Free convolution naturally appears in the study of random Hermitian matrices. Let $A$ be a
$N\times N$ Hermitian matrix with eigenvalues $\{a_i\}_{i=1}^N$. The \emph{empirical measure} of
$A$ is the discrete probability measure on $\mathbb R$ with atoms of weight $\frac{1}{N}$ at points
$a_i$ (cf.\ \eqref{eq_def_uniform_intro}, \eqref{eq_Biane_measure}).

\begin{theorem}[Voiculescu, \cite{Vo3}] For each $N=1,2,\dots$ take two sets of reals
$a(N)=\{a_i(N)\}_{i=1}^N$ and $b(N)=\{b_i(N)\}_{i=1}^N$. Let $\mathcal A(N)$ be the uniformly
random $N\times N$  Hermitian matrix with eigenvalues $a(N)$ and let $\mathcal B(N)$ be the
uniformly random $N\times N$ Hermitian matrix with eigenvalues $b(N)$ such that $\mathcal A(N)$ and
$\mathcal B(N)$ are independent. Suppose that as $N\to\infty$ the empirical measures of $\mathcal
A(N)$ and $\mathcal B(N)$ weakly converge to probability measures $\mes^1$ and $\mes^2$,
respectively. Then the \emph{random} empirical measure of the sum $\mathcal A(N) + \mathcal B(N)$
converges to a \emph{deterministic} measure $\mes^1\boxplus\mes^2$ which is the free convolution of
$\mes^1$ and $\mes^2$. \label{theorem_free_convolution}
\end{theorem}

A link between Theorem \ref{theorem_free_convolution} and decomposition of tensor products of
representations of unitary group is provided by the \emph{semiclassical limit} well-known in the
representation theory, cf.\ \cite{STS}, \cite{He}, \cite{GS}, and also \cite{Ki2} and references
therein. Let us describe this limit in the language of Fourier transforms (or characteristic
functions). For a set of reals $a_1>a_2>\dots>a_N$ let $\mathcal X(a_1,\dots,a_N)$ denote the set
of all $N\times N$ Hermitian matrices with eigenvalues $a_1,\dots,a_N$. The Fourier transform of
the uniform measure on $\mathcal X(a_1,\dots,a_N)$ can be computed using the Harish--Chandra
formula \cite{HC1}, \cite{HC2} (sometimes known also as Itzykson--Zuber \cite{IZ} formula in
physics literature):
\begin{equation}
\label{eq_HC_intro}
 \int_{A\in \mathcal X(a_1,\dots,a_N)} \exp({\rm Trace} (AB)) dA = \frac{\det_{i,j=1,\dots, N} \Bigl(\exp(a_i b_j)\Bigr)}{\prod_{i<j} (a_i-a_j) \prod_{i<j}
 (b_i-b_j)} \prod_{i<j} (j-i),
\end{equation}
where $B$ is a Hermitian matrix with eigenvalues $b_1>b_2>\dots>b_N$.

An analogue of the Fourier transform for the representations of $U(N)$ is their characters, which
are given by the following formula.

\begin{proposition}[Weyl, \cite{Weyl-book}] The value of the character  of irreducible
representation $\pi^{\lambda}$ corresponding to signature $\lambda=(\lambda_1\ge\dots\ge\lambda_N)$
on a unitary matrix $\mathfrak u\in U(N)$ with eigenvalues $u_1,\dots,u_N$ is given by the rational
Schur function:
\begin{equation}
\label{eq_Unitary_character}
 {\rm Trace}(\pi^\lambda(\mathfrak u))= s_\lambda(u_1,\dots,u_N)=\frac{\det_{i,j=1,\dots,N}
 \left(u_i^{\lambda_j+N-j}\right)}{\prod_{1\le i<j\le N} (u_i-u_j)}.
\end{equation}
\end{proposition}
Observe that under the change of variables $u_i=\exp(b_i)$, $i=1,\dots,N$, formulas
\eqref{eq_HC_intro} and \eqref{eq_Unitary_character} look very similar. (This is a manifestation of
a more general phenomena, cf.\ \cite{Ki2}.) However, when we do this change, the denominators
become different. Let us note that on the level of heuristics, the difference between the free
convolution and its quantized version can be traced back to this difference in denominators. The
product $\prod_{i<j} (u_i-u_j)$ in the definition of Schur functions can be written as
$\det_{i,j=1}^N (u_i^{N-j})$ and the appearance of the set $\{N-j\}_{j=1}^N$ in the last formula
predicts the appearance of the uniform measure on $[0,1]$ in \eqref{eq_R_intro}. We are grateful to
Philippe Biane and Grigori Olshanski for this observation.

\smallskip

Comparing \eqref{eq_Unitary_character} with \eqref{eq_HC_intro} one immediately arrives at the
limit relation between them.

\begin{proposition}
\label{prop_semiclassical}
 Fix $N$ and let $\delta>0$ be an auxiliary small parameter. Fix two sequences of reals
 $a_1>\dots>a_N$ and $b_1>\dots>b_N$. Set
 $$
  \lambda_i=\lfloor a_i \delta^{-1}\rfloor,\quad x_i=\exp(\delta b_i),\quad i=1,\dots,N.
 $$
 Then
 $$
  \lim_{\delta\to 0}\left( \frac{s_\lambda(x_1,\dots,x_N)}{s_\lambda(\underbrace{1,\dots,1}_N)}\right) = \int_{A\in \mathcal X(a_1,\dots,a_N)} \exp({\rm Trace} (AB))
  dA,
 $$
 where $B$ is a Hermitian matrix with eigenvalues $b_1>b_2>\dots>b_N$.
\end{proposition}
\noindent {\bf Remark.} The denominator $s_\lambda(\underbrace{1,\dots,1}_N)$ coincides with the
dimension of $\pi^\lambda$ and can be computed by the Weyl's dimension formula (see e.g.\
\cite{Zh})
\begin{equation}
\label{eq_Weyl_dimension}
 s_\lambda(\underbrace{1,\dots,1}_N)=\prod_{1\le i<j\le N} \frac{\lambda_i-i-\lambda_j+j}{j-i}.
\end{equation}

Now observe that in the limit transition of Proposition \ref{prop_semiclassical} the tensor product
of representations becomes the sum of independent Hermitian matrices. Indeed, the character of the
former is the product of characters and the characteristic function (Fourier transform) of the
latter is the product of characteristic functions.

This observation to some extent explains the appearance of the free convolution in the results of
Biane \cite{Biane_U} and Collins--Sniady \cite{CS}: when $\eps(N)$ in \eqref{eq_Biane_measure}
decays faster than a linear function, irreducible representations of $U(N)$ degenerate into
measures on Hermitian matrices; the latter are intrinsically linked to the free convolution, thus,
also the former in this limit regime. Of course, a great amount of work is required to turn this
observation into a rigorous argument, and the proofs in \cite{Biane_U}, \cite{CS} are very delicate
and non-trivial.

On the other hand, we observe that in the limit regime of Theorem \ref{Theorem_main_intro} the
degeneration to random matrices does not happen which is reflected in the new notion of the
\emph{quantized free convolution} replacing the free convolution of random matrices.

One could predict from the above discussion that there should be a limit transition, which
transforms the quantized free convolution into the (conventional) free convolution. This is indeed
true and can be seen from the following asymptotic relation between the Voiculescu $R$--transform
$R_{\mes}(z)$ and its quantized version $R^{quant}_\mes(z)$:
$$
 \lim_{L\to+\infty} \frac{R^{quant}_{\mes*L}\left(\frac{z}{L}\right)}{L}=R_{\mes}(z),
$$
where $\mes*L$ is a probability measure whose value on a measurable set $A$ is defined via
$$
 (\mes*L)(A)=\mes\left(A/L\right),\quad A\subset \mathbb R,\quad L>0.
$$

 As a final remark of this section, let us note that the methods of the present article are
 different from those of \cite{Biane_U}, \cite{CS}. However, it is plausible that our
 methods can be used to give another proof of most of the results of these articles.

\subsection{Perelomov--Popov measures}

In the previous section we were arguing that when $\lambda_i(N)$ grow linearly with $N$ there is no
direct connection between the asymptotics of the measures \eqref{eq_def_uniform_intro} and free
probability. However, this connection can be restored if we change the measure which corresponds to
a signature.

The ``correct'' definition of the measure comes from the work of Perelomov and Popov \cite{PP}
 on the centers of universal enveloping algebras of classical Lie groups. In our
Introduction (in order to keep it short) we present their construction only for the unitary groups,
but parallel stories exist in \cite{PP} for orthogonal and symplectic groups as well. In Section
\ref{Section_results} we present the results for the corresponding measures for all classical
groups.

Let $\mathcal U(\mathfrak{gl}_N)$ denote the complexified universal enveloping algebra of $U(N)$.
This algebra is spanned by generators $E_{ij}$ ($E_{ij}$ as an element of the Lie algebra
$\mathfrak{gl}_N$ can be identified with the $N\times N$ matrix whose single non-zero matrix
element is $1$ at the intersection of the $i$th row and the $j$th column) subject to the relations
$$
 [E_{ij}, E_{kl}]=\delta_{j}^k E_{il} -\delta_i^l E_{kj}.
$$
Let $E(N)\in \mathcal U(\mathfrak{gl}_N)\otimes {\rm Mat}_{N\times N}$ denote the following
$N\times N$ matrix, whose matrix elements belong to $\mathcal U(\mathfrak{gl}_N)$:
$$
E(N)=
\begin{pmatrix} E_{11}& E_{12}& \dots& E_{1N}\\
                E_{21}& \ddots&  &E_{2N}\\
                \vdots&&&\vdots\\
                E_{N1}& E_{N2}&\dots& E_{NN}
\end{pmatrix}
$$
Let $\mathcal Z(\mathfrak{gl}_N)$ denote the center of $\mathcal U(\mathfrak{gl}_N)$ and recall
that each element of $\mathcal Z(\mathfrak{gl}_N)$ acts in an irreducible representation of $U(N)$
(thus, also of $\mathcal U(\mathfrak{gl}_N)$) as a scalar operator.

\begin{theorem}[Perelomov--Popov, \cite{PP}] For $p=0,1,2,\dots$ consider the element
$$
 X_p={\rm Trace}\left( E^p\right)=\sum_{i_1,\dots,i_p=1}^N E_{i_1 i_2} E_{i_2 i_3}\cdots E_{i_p i_1} \in \mathcal U(\mathfrak{gl}_N).
$$
Then $X_p\in\mathcal Z(\mathfrak{gl}_N)$. Moreover, in the irreducible representation parameterized
by $\lambda=(\lambda_1\ge\dots\ge\lambda_N)$ the element $X_p$ acts as scalar $C_p[\lambda]$
\begin{equation}
\label{eq_definition_PP_elements}
 C_p[\lambda]= \sum_{i=1}^N \left(\prod_{j\ne i} \frac{(\lambda_i-i)-(\lambda_j-j)-1}{(\lambda_i-i)-(\lambda_j-j)}\right)
 \left(\lambda_i+N-i\right)^p.
\end{equation}
\end{theorem}
After their discovery, the elements $X_p$ and matrix $E$ have been used in a number of contexts: in
addition to being a nice and useful family of generators (``higher order Casimir operators'') of
the centers of the universal enveloping algebras of the classical Lie groups, they play an
important role in the study of the so-called characteristic identities (cf.\ \cite{Gould} and
references therein) and in the study of Yangians (cf.\ \cite{Olshanski-Nazarov-Molev},
\cite{Molev-book}). They were also to a certain extent already used in the context of the
asymptotic representation theory of symmetric and unitary groups in \cite{Biane_U}, \cite{Biane_S},
\cite{CS}.

For us the elements $X_p$ serve as a motivation to define for a signature $\lambda$ a probability
measure $m_{PP}[\lambda]$ on $\mathbb R$, whose moments would be described by the right-hand side
of \eqref{eq_definition_PP_elements}. Embedding into the definition the rescaling which will be
useful in $N\to\infty$ limit, we arrive at the following formula for the \emph{Perelomov--Popov
measure}
\begin{equation}
\label{eq_PP_intro}
 m_{PP}[\lambda]=\frac{1}{N}\sum_{i=1}^N \left(\prod_{j\ne i} \frac{(\lambda_i-i)-(\lambda_j-j)-1}{(\lambda_i-i)-(\lambda_j-j)}\right) \delta\left(\frac{\lambda_i+N-i}N\right).
\end{equation}

From the probabilistic point of view, the definition of the measure $m_{PP}[\lambda]$ might look
mysterious. Moreover, while the counting measures are related to the combinatorics of lozenge
tilings (see Section \ref{Section_tilings} for the details), we do not yet know any good
combinatorial or probabilistic interpretations for the Perelomov--Popov measures. But, from the
other side, we prove that these measures are much closer than the counting ones related to the free
probability: An analogue of Theorem \ref{Theorem_main_intro} holds for measures $m_{PP}[\lambda]$
with quantized free convolution replaced by the conventional free convolution.

\begin{theorem}
\label{Theorem_main_intro_PP}
 Suppose that $\lambda^1(N),\lambda^2(N)\in\widehat U(N)$, $N=1,2,\dots$, are $2$ sequences of
 signatures which satisfy a technical assumption of Definition \ref{definition_regularity} and
 such that
 $$
  \lim_{N\to\infty} m_{PP}[\lambda^i(N)]=\mes^i,\text{ (weak convergence), } i=1,2.
 $$
 Let $\pi(N)=\pi^{\lambda^1(N)}\otimes \pi^{\lambda^2(N)}$. Then as $N\to\infty$ random measures $m_{PP}[\rho^{\pi(N)}]$ converge in the sense of
moments, in probability to a deterministic measure $\mes^1 \boxplus \mes^2$ which is the free
convolution of $\mes^1$ and $\mes^2$.
\end{theorem}

Theorem \ref{Theorem_main_intro_PP} leads us to conjecture that the images of matrices $E$ in
different representations are asymptotically free, since this would agree with their sum being
asymptotically related to the free convolution (see \cite{VDN}, \cite{NS} for the details on the
freeness). Let us give more definitions to state a conjecture.

Let us take two signatures $\lambda^1(N),\lambda^2(N)\in\widehat U(N)$, recall that
$\pi^{\lambda^1(N)}$ and $\pi^{\lambda^2(N)}$ are the corresponding representations and let
$V_{\lambda^1(N)}$, $V_{\lambda^2(N)}$ denote the spaces of these representations. Consider the
complex algebra
$$
 \mathbf{A}(N)={\rm End}(V_{\lambda^1(N)})\otimes {\rm End}(V_{\lambda^2(N)})\otimes {\rm
 Mat}_{N\times N}
$$
equipped with the usual normalized trace
$$
\dfrac{{\rm Trace}_{V_{\lambda^1(N)}}\otimes {\rm Trace}_{V_{\lambda^2(N)}}\otimes {\rm
Trace}_N}{\dim(V_{\lambda^1(N)})\cdot \dim(V_{\lambda^2(N)})\cdot N}.
$$
$\mathbf{A}(N)$ can be viewed as a \emph{non-commutative probability space}, cf.\ \cite{NS}.
Further define the element $E(\lambda^1(N))\in \mathbf{A}(N)$ by replacing $E_{ij}$,
$i,j=1,\dots,N$ in the definition of $E(N)$ by $\pi^{\lambda^1(N)}(E_{ij})\otimes Id$, where $Id$
is the identical operator. Similarly define $E(\lambda^2(N))\in\mathbf{A}(N)$ by replacing
$E_{ij}$ by $Id \otimes \pi^{\lambda^2(N)}(E_{ij})$.

\begin{conjecture}
\label{conjecture}
  Suppose that $\lambda^1(N),\lambda^2(N)\in\widehat U(N)$, $N=1,2,\dots$, are $2$ sequences of
 signatures which satisfy a technical assumption of Definition \ref{definition_regularity} and
 such that
 $$
  \lim_{N\to\infty} m_{PP}[\lambda^i(N)]=\mes^i,\quad i=1,2.
 $$
 Then as $N\to\infty$ the elements $\frac{1}{N}E(\lambda^1(N))$ and $\frac{1}{N}E(\lambda^2(N))$ of non-commutative
 probability spaces $\mathbf{A}(N)$ become asymptotically free.
\end{conjecture}
We refer to \cite[Lecture 5]{NS} for the definition of the asymptotical freeness.

Note that in the superlinear limit regime discussed in Section \ref{Section_superlinear} an
analogue of Conjecture \ref{conjecture} was proved by Biane \cite{Biane_U}.

\subsection{Markov--Krein correspondence}
\label{section_MK}

It is natural to ask about the exact relationship between the free convolution and its quantized
version that we study in the present article. An asymptotic relation was explained in Section
\ref{Section_superlinear}, and a non-asymptotic one is provided by the following theorem.

\begin{theorem}
\label{theorem_intertwining}
 For every probability measure $\rho$ on $\mathbb R$ which has  compact support, is absolutely
 continuous with respect to the Lebesgue measure, and has bounded by $1$ density, there exists a probability
 measure $Q(\rho)$ with compact support on $\mathbb R$, such that
$$ \exp \left(- \sum_{k=0}^{\infty} s_k z^{k+1} \right)=1 - \sum_{k=0}^{\infty} c_k z^{k+1},$$
where $s_k$ and $c_k$ are the moments of $\rho$ and $Q(\rho)$, respectively, i.e.\
$$
 s_k=\int_{\mathbb R} x^k \rho(dx),\quad c_k=\int_{\mathbb R} x^k Q(\rho)(dx), \quad
 k=0,1,2,\dots.
$$
The operation $Q$ intertwines the free convolution and its quantized version, i.e.\ for any two
$\rho_1$, $\rho_2$, as above, we have
$$
 Q(\rho_1)\boxplus Q(\rho_2)=Q(\rho_1\otimes\rho_2).
$$
\end{theorem}

The non-trivial part of Theorem \ref{theorem_intertwining} is the existence of the
map $Q(\cdot)$, while the intertwining property is a simple corollary of the
definitions of functions $R_\mes(z)$ and $R^{quant}_\mes(z)$. One way to prove the
existence of $Q(\rho)$ (see Theorem \ref{theorem_moment_link_to_PP}) is through the
limit transition in the formulas of \cite{PP}, \cite{Popov}, \cite{Popov2} linking
the moments of counting measures $m[\lambda]$ with those of Perelomov--Popov
measures $m_{PP}[\lambda]$.

The operation $\rho\mapsto Q(\rho)$ is a close relative of the Markov--Krein
correspondence. In the context of the asymptotic representation theory of symmetric
groups this correspondence was introduced and studied by Kerov (see \cite[Chapter
IV]{Kerov-book}), but its origins go back to the Hausdorff moment problem and Markov
moment problem (a recent review can be found in \cite{DF}). The former asks about
necessary and sufficient conditions for a sequence $\{a_k\}_{k=0,1,\dots}$ to be a
sequence of moments of a measure. And the latter asks about the necessary and
sufficient conditions for a sequence $\{b_k\}_{k=0,1,\dots}$ to be a sequence of
moments of a measure, which is absolutely continuous with respect to the Lebesgue
measure and whose density is bounded by $1$. The relation between these two problems
is explained in the following theorem.

\begin{theorem}[Ahiezer--Krein \cite{AK}, Krein--Nudelman \cite{KN}] \label{theorem_Markov_Krein}
Let $\rho$ be a finite measure with compact support on $\mathbb R$, which is
absolutely continuous with respect to the Lebesgue measure and whose density is
bounded by $1$. Then there exists a finite measure $MK(\rho)$ with compact support
on $\mathbb R$ such that
\begin{equation}
\label{eq_Markov_Krein}
 \exp\left(\sum_{k=0}^{\infty} b_k z^{k+1}\right)=1+\sum_{k=0}^{\infty} a_k z^{k+1},
\end{equation}
where $\{a_k\}_{k=0,1,\dots}$ and $\{b_k\}_{k=0,1,\dots}$ are the moments of
$MK(\rho)$ and $\rho$, respectively. Moreover, $MK(\cdot)$ is a bijection, i.e.\ for
any finite measure $\nu$ there exists a unique $\rho$ such that $MK(\rho)=\nu$.
\end{theorem}
Comparing Theorems \ref{theorem_intertwining} and \ref{theorem_Markov_Krein} one
immediately sees that
\begin{equation}
\label{eq_link_to_Markov_Krein}
 Q(\rho)=\bigl( (-)^* \circ MK  \circ(-)^*\bigr)(\rho),
\end{equation}
where $(-)^*$ is the reflection of a measure with respect to the origin, i.e.\ for any measurable
$A$
$$
 (-)^*(\rho) (A)=\rho(-A),\quad A\subset \mathbb R,
$$
The formula \eqref{eq_link_to_Markov_Krein} reduces $Q(\cdot)$ to $MK(\cdot)$. In
particular, this gives another way to prove Theorem \ref{theorem_intertwining}. In
addition it shows that the correspondence $\rho\mapsto Q(\rho)$ is bijective.

\smallskip

We should note that although Theorem \ref{theorem_intertwining} formally reduces the
quantized free convolution to the conventional free convolution, it still makes
sense to distinguish these two operations because due to complexity of
\eqref{eq_Markov_Krein}, the Markov--Krein correspondence $MK(\cdot)$ is a very
non-trivial and highly non-linear operation on measures (see, however, \cite[Chapter
IV, Section 4]{Kerov-book} where an elegant probabilistic algorithm for sampling
from $MK(\rho)$ is proposed).

\subsection{Our methods}

There are three main ingredients in the proofs of the results of the present article.

The first one is the method of analysis of the measures appearing in the decomposition of
representations into irreducible components using the application of relatively simple
differential operators to the characters of these representations. One way to view the operators
we use is that they are \emph{radial parts} of the differential operators in the centers of
universal enveloping algebras for classical Lie groups. One very important feature that we observe
here is that for the asymptotic analysis we need only the values of characters and their
derivatives with all but finitely many variables set to $1$.

The second ingredient is the asymptotic expansion for the characters of classical Lie groups as the
rank of the group goes to infinity obtained by one of the authors and Panova in \cite{GP} (and
which is a generalization of earlier results of Guionnet and Maida \cite{GM} on matrix integrals).
In particular, for symplectic and orhtogonal groups we use an interesting finite $N$ relation
between their normalized characters and those for $U(N)$ (see Propositions
\ref{Proposition_Schur_Symplectic_1}-\ref{Proposition_Schur_Ortho_2}).

Finally, our analysis of Perelomov--Popov measures also uses the formulas of \cite{PP},
\cite{Popov}, \cite{Popov2} relating the moments of these measures to the moments of counting
measures.

\subsection{Acknowledgements}

We would like to thank Alexei Borodin and Grigori Olshanski for many valuable discussions. A.~B.\
was partially supported by Simons Foundation-IUM scholarship, by ``Dynasty'' foundation, by
Moebius Foundation for Young Scientists and by the RFBR grant 13-01-12449.  V.~G.\ was partially
supported by the RFBR-CNRS grant 11-01-93105.

\section{Setup and results}

\label{Section_Setup}

\subsection{Preliminaries}
Let $G(N)$ be one of the classical real Lie groups of rank $N$, i.e.\ $G(N)$ is either unitary
group $U(N)$ or orthogonal group $SO(2N)$, or orthogonal group $SO(2N+1)$, or symplectic group
$Sp(2N)$. These groups correspond to the root systems $A$, $D$, $B$ and $C$, respectively, and we
will use both groups and root systems in our notations. Thus, the letter $G$ should be also
understood as either $A$, $B$, $C$ or $D$.

Irreducible representations of $G(N)$ are parameterized by their highest weights, which are
\emph{signatures}, i.e.\ $N$--tuples of integers $\lambda_1\ge\lambda_2\ge\dots\ge\lambda_N$. When
$G(N)=Sp(2N)$ or $SO(2N+1)$, one should also assume $\lambda_N\ge 0$; when $G(N)=SO(2N)$, last
coordinate $\lambda_N$ can be negative, but $\lambda_{N-1}\ge|\lambda_N|$, see e.g.\ \cite{Zh},
\cite{FH}. Let $\hG(N)$ denote the set of signatures parameterizing irreducible representations of
$G(N)$ and let $\pi^\lambda$, $\lambda\in\hG(N)$ denote the irreducible representation
corresponding to $\lambda$.

It is convenient for us to encode signatures by probability measures on $\mathbb R$. We will use
two different sets of measures.

\begin{definition}
The counting measure $m^G[\lambda]$ corresponding to a signature
$\lambda=(\lambda_1\ge\lambda_2\ge\dots\ge\lambda_N)$ is defined through
$$
 m^A[\lambda]=\frac{1}{N}\sum_{i=1}^N \delta\left(\frac{\lambda_i+N-i}N\right)
$$
for the unitary groups and for $G=B,C,D$ we set
$$
 m^G[\lambda]=\frac{1}{2N} \sum_{i=1}^N \left( \delta\left(\frac{\lambda_i+2N-i}{2N}\right) +
 \delta\left(\frac{i-\lambda_i}{2N}\right)\right).
$$
\end{definition}
\noindent {\bf Remark.} In principle, we could have kept the same definition of the counting
measure for all root systems. However, our current definition is consistent with the lozenge
tilings interpretations of Section \ref{Section_tilings}. Also this definition makes the statement
of Theorem \ref{Theorem_linearization} independent of the root system.

\smallskip

The definition of the \emph{Perelomov--Popov measure} $m^G_{PP}[\lambda]$ is a bit more delicate.
For a signature $\lambda$ let $\lambda^{(i+)}$ ($\lambda^{(i-)}$), $i=1,\dots,N$ denote the
sequence of integers obtained from $\lambda$ by increasing (decreasing) the $i$th coordinate by
$1$. Note that $\lambda^{(i\pm)}$ might be not a signature. Let $\dim(\lambda)$, $\lambda\in\hG(N)$
denote the dimension of the irreducible representation $\pi^\lambda$ and let
$\dim({\lambda^{(i\pm)}})$ be the dimension of $\pi^{\lambda^{(i\pm)}}$ if $\lambda^{(i\pm)}$ is a
signature, and $0$ otherwise.

\begin{definition} \label{Definition_PP} The Perelomov--Popov measure $m^G_{PP}[\lambda]$ corresponding to $\lambda\in\hG(N)$ is defined through
$$
 m_{PP}^A[\lambda]=\frac{1}{N}\sum_{i=1}^N \frac{\dim(\lambda^{(i-)})}{\dim(\lambda)}
 \delta\left(\frac{\lambda_i+N-i}{N}\right),
$$
\begin{multline*}
m_{PP}^{B}[\lambda] = \frac{1}{2N+1} \Biggl[ \sum_{i=1}^{N} \Biggl( \frac{\dim
({\lambda^{(-i)}})}{\dim ({\lambda})} \delta\left(\frac{ \lambda_i +2N-i}{2N+1}\right) \\+
\frac{\dim ({\lambda^{(+i)}})}{\dim ({\lambda})} \delta\left(\frac{ i-1-\lambda_i}{2N+1}\right)
\Biggr) + \delta\left(\frac{N}{2N+1}\right) \Biggr],
\end{multline*}
$$
m_{PP}^{C}[\lambda] = \frac{1}{2N} \sum_{i=1}^{N} \left( \frac{\dim ({\lambda^{(-i)}})}{\dim
({\lambda})} \delta\left(\frac{ \lambda_i + 2N+1 - i}{2N}\right) + \frac{\dim
({\lambda^{(+i)}})}{\dim ({\lambda})} \delta\left(\frac{i-1 -\lambda_{i}}{2N}\right) \right),
$$
$$
m_{PP}^{D}[\lambda] = \frac{1}{2N} \sum_{i=1}^{N} \left( \frac{\dim ({\lambda^{(-i)}})}{\dim
({\lambda})} \delta\left(\frac{ \lambda_i +2N-1-i}{2N}\right) + \frac{\dim ({\lambda^{(+i)}})}{\dim
({\lambda})} \delta\left(\frac{i-1 -\lambda_i}{2N}\right) \right).
$$
\end{definition}

{\bf Remark.} Using the Weyl dimension formula the constants in the definition of the
Perelomov--Popov measure can be computed. The result for different groups is similar. For instance,
\begin{equation}
\label{eq_PP_main_text}
 m_{PP}^A[\lambda]=\frac{1}{N}\sum_{i=1}^N \left(\prod_{j\ne i} \frac{(\lambda_i-i)-(\lambda_j-j)-1}{(\lambda_i-i)-(\lambda_j-j)}\right) \delta\left(\frac{\lambda_i+N-i}N\right).
\end{equation}

\smallskip

Definition \ref{Definition_PP} is motivated by the work of Perelomov and Popov \cite{PP} on the
center of the universal enveloping algebra of semisimple Lie groups. They produced a distinguished
set of elements $C_p^G$, $p=1,2,\dots$ which generate the center of the universal enveloping
algebra of $G$. These elements were further used by several authors, see e.g.\
\cite{Olshanski-Nazarov-Molev}, \cite{Molev-book}, \cite{Gould}, \cite{Biane_U}, \cite{Biane_S},
\cite{CS}. Since each $C_p^G$ belongs to the center of the universal enveloping algebra, it acts as
a \emph{constant} $C_p^G[\lambda]$ in an irreducible representation parameterized by $\lambda$. The
relation between $C_p^G[\lambda]$ and the measures $\mu_{PP}^G[\lambda]$ is explained in the
following theorem.

\begin{theorem}[{\cite[Eq.\ 71]{PP}}] For $G$ being either unitary, orthogonal or symplectic group, we have
 $$
  C_p^G[\lambda]=(\hat N)^{p+1} \int_{\mathbb R}   x^p\, m^G_{PP}[\lambda](dx),
 $$
 where $\hat N=N$ for the unitary group $U(N)$, $\hat N=2N$ for $Sp(2N)$ and $SO(2N)$, $\hat N=2N+1$
 for $SO(2N+1)$.
\end{theorem}

In particular, the definitions of \cite{PP} imply that $C_0^G$ is just an identical operator in
each representation and, thus, $m^G_{PP}[\lambda]$ is a probability measure. This can be also
checked independently.
\begin{proposition}
 For any signature $\lambda\in\hG(N)$, both $m^G[\lambda]$ and $m^{G}_{PP}[\lambda]$ are probability measures on
 $\mathbb R$.
\end{proposition}
\begin{proof}
 For $m^G[\lambda]$ this is immediate. For $m^A_{PP}[\lambda]$ note that
 $$
  \sum_{i=1}^N \prod_{j\ne i}
  \frac{(\lambda_i-i)-(\lambda_j-j)-1}{(\lambda_i-i)-(\lambda_j-j)}=\frac{1}{\mathbf{V}(\lambda)}
  \sum_{i=1}^N T_i(\mathbf{V}(\lambda)),
 $$
 where $\mathbf{V}(\lambda)=\prod_{i<j}
 ((\lambda_i-i)-(\lambda_j-j))$ and $T_i$ is the operator which decreases $\lambda_i$ by $1$.
 Moreover, $\sum_{i=1}^N T_i(\mathbf{V}(\lambda))$ is a skew-symmetric polynomial in $\lambda_i-i$ of degree
 $N(N-1)/2$, therefore, it is proportional to $\mathbf{V}(\lambda)$. Comparing the leading terms, we get $\sum_{i=1}^N
 T_i(\mathbf{V}(\lambda))=N \mathbf{V}(\lambda)$. The proof for $m^{B}_{PP}[\lambda]$, $m^{C}_{PP}[\lambda]$, $m^{D}_{PP}[\lambda]$
 is similar.
\end{proof}

\subsection{Main results}

\label{Section_results}

In this section we state the main results which yield that \emph{random} measures corresponding to
restrictions and tensor products of representations of $G(N)$ are asymptotically
\emph{deterministic}, thus showing a form of the Law of Large Numbers. The proofs are given in
Sections \ref{Section_asymptotics_of_measures}-\ref{Section_appendix}.

In our asymptotic results we are going to make the following technical assumption on the behavior
of signatures $\lambda(N)$ as $N$ becomes large. It is plausible that this assumption can be
weakened, but we do not address this question in the present article.

\begin{definition}
\label{definition_regularity}
 A sequence of signatures $\lambda(N)\in\hG(N)$ is called \emph{regular}, if there exists a
 piecewise--continuous function $f(t)$ and a constant $C$ such that
\begin{equation}
\label{eq_reg1}
 \lim_{N\to\infty} \frac{1}{N}\sum_{j=1\dots,N} \left|\frac{\lambda_j(N)}{N}-f(j/N)\right|=0
\end{equation}
and
\begin{equation}
\label{eq_reg2} \left|\frac{\lambda_j(N)}{N}-f(j/N)\right|<C,\quad \quad j=1,\dots,N,\quad
N=1,2,\dots.
\end{equation}
\end{definition}
\noindent{\bf Remark.} Informally, the condition \eqref{eq_reg1} means that scaled by $N$
coordinates of $\lambda(N)$ approach a limit profile $f$. The restriction that $f(t)$ is
piecewise--continuous is reasonable, since $f(t)$ is a limit of monotonous functions and, thus, is
monotonous (therefore, we only exclude the case of countably many points of discontinuity for $f$).
We use condition \eqref{eq_reg2} since it guarantees that all the measures which we assign to
signatures and their limits have (uniformly) compact supports --- thus, these measures are uniquely
defined by their moments.

\medskip

\begin{lemma}
 Suppose that $\lambda(N)\in\hG(N)$, $N=1,2,\dots$ is a regular sequence. Then the measures $m^G[\lambda(N)]$ and $m^G_{PP}[\lambda(N)]$ converge as $N\to\infty$
 (weakly and in the sense of moments)
  to probability measures with compact support.
\end{lemma}
\begin{proof}
 For measures  $m^G[\lambda(N)]$ this is immediate from the definitions. For measures
 $m^G_{PP}[\lambda(N)]$ this follows from Theorem \ref{theorem_moment_link_to_PP} below.
\end{proof}

Let $\lambda^1,\dots,\lambda^k$ be elements of $\hG(N)$ and let
$\pi^{\lambda^1},\dots,\pi^{\lambda^k}$ be the corresponding irreducible representations of $G(N)$.
For $\mu\in\hG(N)$ set
$$
 P^{\lambda^1,\dots,\lambda^k}(\mu)=\frac{c^{\lambda^1,\dots,\lambda^k}_{\mu}
 \dim^G(\mu)}{\dim^G(\lambda^1)\cdots\dim^G(\lambda^k)},
$$
where $\dim^G$ stays for the dimension of the corresponding irreducible representation and
$c^{\lambda^1,\dots,\lambda^k}_{\mu}$ is multiplicity of $\pi^\mu$ in the (Kronecker) tensor
product $\pi^{\lambda^1}\otimes\dots\otimes\pi^{\lambda^k}$. In other words, $P(\mu)$ is relative
dimension of the isotypic component $\mu$ in the tensor product. Since
$P^{\lambda^1,\dots,\lambda^k}(\mu)\ge 0$ and $\sum_\mu P^{\lambda^1,\dots,\lambda^k}(\mu)=1$,
these numbers define a probability measure on $\hG(N)$ which we denote
$\rho^{\lambda^1\otimes\cdots\otimes\lambda^k}$.

In a similar way, let $\lambda\in\hG(N)$ and $0<\alpha<1$. For $\mu\in\hG(\lfloor \alpha N\rfloor)$
define
$$
 P^{\alpha,\lambda}(\mu)=\frac{c^{\lambda}_{\mu}
 \dim^G(\mu)}{\dim^G(\lambda)},
$$
where $c^\lambda_\mu$ is multiplicity of $\pi^\mu$ in the restriction of $\pi^\lambda$ on
$G(\lfloor \alpha N\rfloor)\subset G(N)$ (embedded as the subgroup fixing last basis vectors). The
numbers $P^{\alpha,\lambda}(\mu)$ define a probability measure on $\hG(\lfloor \alpha N\rfloor )$
which we denote $\rho^{\alpha,\lambda}$.

Recall that each element $\lambda\in\hG(N)$ defines a probability measure $m^G[\lambda]$ on
$\mathbb R$. Thus, if $\lambda$ is random and distributed according to $\rho$, then $m^G[\lambda]$
becomes a \emph{random} probability measure on $\mathbb R$. Somewhat abusing the notations we
denote this random measure through $m^G[\rho]$. We similarly define $m^G_{PP}[\rho]$.

\begin{theorem}[Law of Large Numbers for counting measures]
\label{Theorem_main}
 Suppose that $\lambda^1(N),\dots,\lambda^k(N)\in\hG(N)$, $N=1,2,\dots$, are $k$ regular sequences of signatures
 such that
 $$
  \lim_{N\to\infty} m^G[\lambda^i(N)]=\mes^i,\quad i=1,\dots,k.
 $$
 Then as $N\to\infty$
\begin{itemize}
\item Random measures $m^G[\rho^{\lambda^1(N)\otimes\cdots\otimes\lambda^k(N)}]$ converge in the sense of moments, in probability to a
deterministic measure which we denote $\mes^1 \otimes \mes^2 \otimes \dots \otimes\mes^k$.
 \item Random measures  $m^G[\rho^{\alpha,\lambda^1(N)}]$ converge  in the sense of moments, in probability to a deterministic measure which we denote $pr^\otimes_\alpha(\mes^1)$.
 \end{itemize}
\end{theorem}

\begin{theorem}[Law of Large Numbers for Perelomov--Popov measures]
\label{Theorem_main_PP}
 Suppose that $\lambda^1(N),\dots,\lambda^k(N)\in\hG(N)$, $N=1,2,\dots$, are $k$ regular sequences of signatures,
 such that
 $$
  \lim_{N\to\infty} m^G_{PP}[\lambda^i(N)]=\mes^i_{PP},\quad i=1,\dots,k.
 $$
 Then as $N\to\infty$
\begin{itemize}
\item Random measures $m^G_{PP}[\rho^{\lambda^1(N)\otimes\cdots\otimes\lambda^k(N)}]$ converge  in the sense of moments,
in probability to a deterministic measure which we denote $\mes^1_{PP} \boxplus \mes^2_{PP} \boxplus \dots \boxplus \mes^k_{PP}$.
 \item Random measures  $m^G_{PP}[\rho^{\alpha,\lambda^1(N)}]$ converge  in the sense of moments,
 in probability to a deterministic measure which we denote $pr^\boxplus_\alpha(\mes^1_PP)$.
 \end{itemize}
\end{theorem}
\noindent{\bf Remark 1.} By the convergence ``in the sense of moments,
 in probability'' we mean that for each $n=0,1,2,\dots$ the $n$th moment of the random measure
 converges in probability to the $n$th moment of the deterministic limit measure. In our setting,
 this implies also weak convergence in probability.

\noindent{\bf Remark 2.} For simplicity we formulate the statements for the restrictions to
$G(M)\subset G(N)$, but one can readily produce similar results for the restrictions to $G(M)\times
G(N-M)\subset G(N)$.

 \smallskip

The operations on the measures which appear in Theorems \ref{Theorem_main} and
\ref{Theorem_main_PP} are best described using certain generating functions.

Given a probability measure $\mes$ with compact support
 set
 $$S_\mes(z)=z+M_1(\mes)z^2+z^3 M_2(\mes) +\dots$$
 to be the generating function of the moments of $\mes$: $M_k(\mes)=\int_{\mathbb R} x^k \mes(dx)$. Define $S_\mes^{(-1)}(z)$ to be the inverse series to $S_\mes(z)$, i.e.\ such that
 $$
  S_\mes^{(-1)}\bigl( S_\mes(z)\bigr)=S_\mes\bigl(S_\mes^{(-1)}(z)\bigr)=z.
 $$
Further, set
\begin{equation}
\label{eq_Voiculescu_R_disc}
 R^{quant}_\mes(z)= \frac{1}{S_\mes^{(-1)}(z)} -\frac{1}{1-e^{-z}}.
\end{equation}
\begin{equation}
\label{eq_Voiculescu_R}
 R_\mes(z)= \frac{1}{S_\mes^{(-1)}(z)} -\frac{1}{z}.
\end{equation}
Note that $R^{quant}_\mes(z)$ and $R_\mes(z)$ are power series in $z$. The function $R_\mes(z)$ is
well-known in the free probability theory under the name of \emph{Voiculescu $R$--transform}, cf.\
\cite{VDN}, \cite{NS}. Immediately from the definitions we have the following relation between our
functions:
$$
 R^{quant}_\mes(z)=R_\mes(z)+\frac{1}{z}-\frac{1}{1-e^{-z}}=R_\mes(z)-R_{u[0,1]}(z),
$$
where $u[0,1]$ is the uniform measure on the interval $[0,1]$.

\begin{theorem} \label{Theorem_linearization} In the notations of Theorems \ref{Theorem_main}, \ref{Theorem_main_PP} we have
\begin{equation}
\label{eq_linear_final_1}
 R^{quant}_{\mes^1 \otimes \mes^2 \otimes \dots \otimes\mes^k}(z)= R^{quant}_{\mes^1}(z)+\dots
 +R^{quant}_{\mes^k}(z),
\end{equation}
\begin{equation}
\label{eq_linear_final_2}
 R^{quant}_{pr^\otimes_\alpha(\mes)}(z)= \frac{1}{\alpha} R^{quant}_{\mes}(z),
\end{equation}
\begin{equation}
\label{eq_linear_final_1_PP}
 R_{\mes^1 \boxplus \mes^2 \boxplus \dots \boxplus\mes^k}(z)= R_{\mes^1}(z)+\dots
 +R_{\mes^k}(z),
\end{equation}
\begin{equation}
\label{eq_linear_final_2_PP}
 R_{pr^\boxplus_\alpha(\mes)}(z)= \frac{1}{\alpha} R_{\mes}(z).
\end{equation}
\end{theorem}

In particular, Theorem \ref{Theorem_linearization} implies that the operations
 $(\mes^1,\dots,\mes^k)\to\mes^1 \boxplus \dots \boxplus\mes^k$ and $\mes\to pr^\boxplus_\alpha(\mes)$ are
 \emph{free convolution} and \emph{free projection} (or free compression with a free projector), respectively, cf.\ \cite{VDN}, \cite{NS}.

\section{Corollaries, connections and reformulations}

\label{Section_discussion}

The aim of this section is to link Theorems \ref{Theorem_main}, \ref{Theorem_main_PP},
\ref{Theorem_linearization} to three topics: asymptotics of operations on irreducible
representations of symmetric groups, random lozenge tilings, characters of the
infinite--dimensional unitary group $U(\infty)$.

\subsection{Symmetric groups}
\label{Section_symmetric}

As the reader might have noticed in Section \ref{Section_History}, the stories for symmetric and
unitary groups are parallel. Moreover, recently Borodin and Olshanski \cite{BO-US} explained how
the asymptotic representation theory corresponding to the infinite-dimensional unitary group
$U(\infty)$ can be degenerated into the one for $S(\infty)$. Thus, it comes as no surprise that
there exists a limit transition from the constructions of Section \ref{Section_Setup} into the
objects related to symmetric group that we will now describe.

Recall that irreducible representations of $S(n)$ are parameterized by partitions of $n$
(equivalently, Young diagrams with $n$ boxes). We need one particular way to associate a
probability measure on $\mathbb R$ to a Young diagram $\lambda$, which is known as \emph{Kerov's
transition measure} \cite{Ker93}, \cite{Kerov-book}. Its definition might look more complicated
than the ones we had for the unitary group, but it turns out to be very useful in various contexts,
 cf.\ \cite[Chapter IV]{Kerov-book}, \cite{Biane_S}, \cite{Olsh_Jack}, \cite{Buf12}. We rotate the Young diagram as shown in Figure \ref{Figure_Sym} and set
$x_i$, $i=1,\dots,k$ to be the the horizontal coordinates of its local minima (inner corners) and
$y_i$, $i=1,\dots,k-1$ to be the coordinates of its local maxima (outer corners).

\begin{figure}[h]
\begin{center}
 {\scalebox{1.5}{\includegraphics{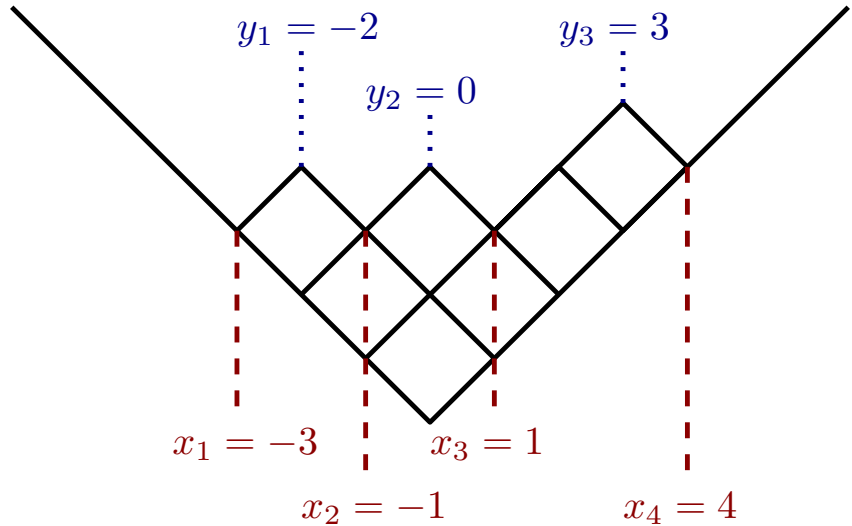}}}
 \caption{Young diagram with $7$ boxes and $4$ rows $(3,2,1,1)$, minima at points $-3,-1,1,4$ and maxima at points $-2,0,3$.
 \label{Figure_Sym}}
\end{center}
\end{figure}

\begin{definition}[Kerov]
 Transition measure $m_{Kerov}[\lambda]$ of Young diagram $\lambda$ is defined as
 $$
  m_{Kerov}[\lambda]=\sum_{i} \frac{\prod_j (x_i-y_j)}{\prod_{j\ne i} (x_i-x_j)}\delta(x_i).
 $$
\end{definition}
\noindent {\bf Remark.} The name for the measure comes from its relation to the Plancherel growth
model, see \cite[Chapter IV]{Kerov-book}.

\smallskip

The relation between Perelomov--Popov measure and Kerov transition measure is explained in the
following statement, in which $\overline {m}_{PP}^{A}[\lambda]$ stays for the pusforward of the
Perelomov--Popov measure under the map $x\mapsto Nx$ (in other words, we remove the denominator
$N$ in delta-functions in \eqref{eq_PP_main_text}.)

\begin{proposition}
\label{Proposition_to_Kerov}
 Take a Young diagram $\lambda$ with non-zero rows $\lambda_1\ge\lambda_2\ge\dots\ge\lambda_k>0$.
 For $N\ge k$ define a signature $\lambda(N)\in\widehat U(N)$ through
 $$
 \lambda(N)=(\underbrace{0,\dots,0}_{N-k},-\lambda_k,-\lambda_{k-1},\dots,-\lambda_1).
 $$
 Then
 $$
  \lim_{N\to\infty} \overline{m}^A_{PP}[\lambda(N)]=m_{Kerov}[\lambda].
 $$
\end{proposition}
\begin{proof} Straightforward computation. In implicit form this statement can be also found in
\cite[Proposition 7.2]{Biane_S}.
\end{proof}

In his study of the operations on irreducible representations of symmetric groups, Biane
\cite{Biane_S} was investigating the behavior of Kerov transition measures as $n\to\infty$. In
particular, for the \emph{outer products} of the representations of $S(n)$ he obtained an exact
analogue of our Theorem \ref{Theorem_main_intro_PP}, which also involves the free convolution.
Proposition \ref{Proposition_to_Kerov} and the well-known fact that the decomposition into
irreducible components of outer products of representations of $S(n)$ and of (Kronecker) tensor
products of representations of $U(N)$ are essentially governed by the same Littlewood--Richardson
coefficients, shows a relation between two results, however, formally, neither is implied by
another.

\subsection{Restrictions and lozenge tilings}
\label{Section_tilings}

The developments of this section are based on the remarkable connection between the
Gelfand--Tsetlin bases in irreducible representations and the enumeration of \emph{lozenge tilings
of planar domains}, cf.\ \cite[Section 2]{CLP} and also \cite{BG_lectures}, \cite{BP_lectures}.

\medskip

Consider a strip of width $N$ drawn on the regular triangular lattice, whose left boundary is
vertical line $x=0$ and right boundary is vertical line $x=N$ with $N$ triangles with vertical
coordinates $\mu_1>\mu_2>\dots>\mu_N$ sticking out of it, as shown in Figure
\ref{Fig_polyg_domain}. Note that if we write $\mu_i=\lambda_i+N-i$, then $\lambda$ is a signature
of size $N$; let us assume that $\lambda_N=0$. We are interested in \emph{tilings} of this domain
(strip) with rhombi (``\emph{lozenges}'') of 3 types, where each rhombus is a union of 2
elementary triangles. In Figure \ref{Fig_polyg_domain} we color one of the type of lozenges
(``horizontal'') in blue and other two types are kept white. Due to combinatorial constraints the
tilings of such domain are in bijection with tilings of a polygonal domain, as shown on the right
panel of Figure \ref{Fig_polyg_domain}. In particular, there are finitely many such tilings, let
$\Upsilon_\lambda$ denote the uniformly random tiling of the domain encoded by a signature
$\lambda$.

\begin{figure}[h]
\begin{center}
 \noindent{\scalebox{0.7}{\includegraphics{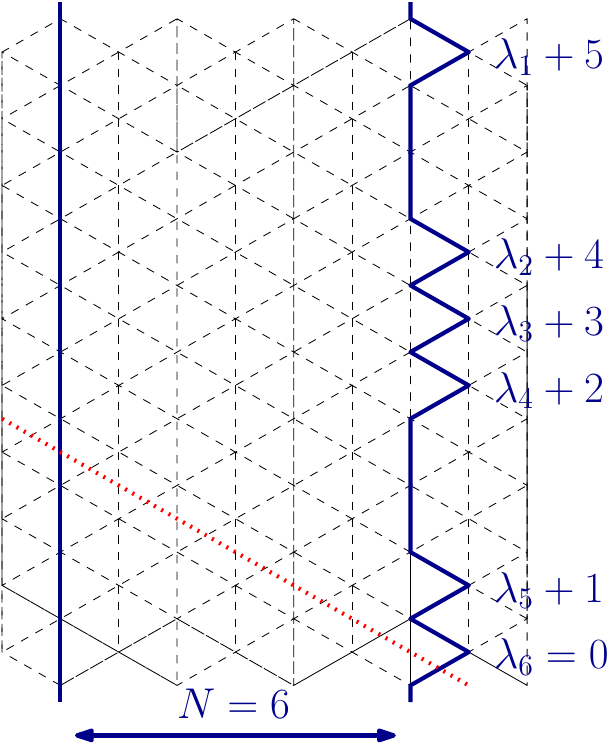}}} \hskip 1cm {\scalebox{0.7}{\includegraphics{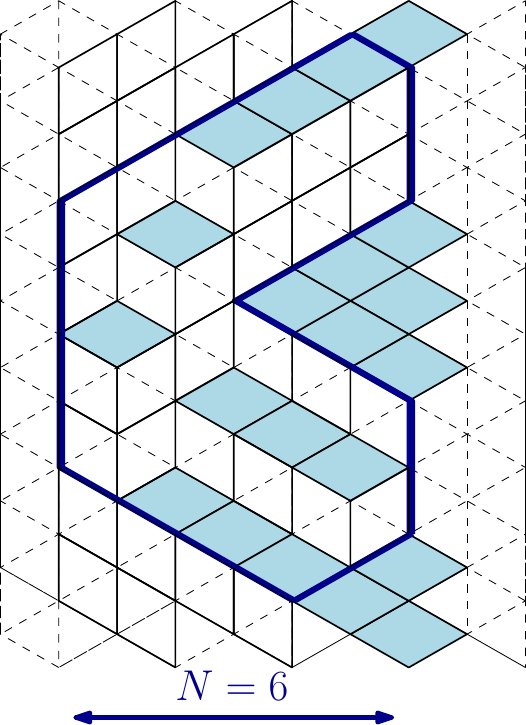}}}
 \caption{Domain encoded by signature $4\ge 2\ge2\ge2\ge0\ge 0$ (left panel, vertical coordinates of lozenges are counted from the red dotted line)
  and a lozenge tiling of corresponding polygonal domain (right panel).
 \label{Fig_polyg_domain}}
\end{center}
\end{figure}

There was a great interest in random lozenge tilings of  planar domains and their asymptotic
properties as the size of the domain goes to infinity in the last 15 years, with many fascinating
results, see e.g.\ \cite{CKP}, \cite{OR}, \cite{PS}, \cite{KOS}, \cite{KO-Burgers}, \cite{Ke},
\cite{BF}, \cite{BG},  \cite{BGR}, \cite{Petrov-curves}, \cite{GP}, \cite{Sevak} and many others.

In particular, the following limit shape theorem is a particular case of a more general statement
of Cohn--Kenyon--Propp \cite{CKP}, Kenyon--Okounkov--Sheffield \cite{KOS}. For a point $(x_0,y_0)$
in the strip of Figure \ref{Fig_polyg_domain} on $(x,y)$--plane, define the value of \emph{the
height function} $H(x_0,y_0)$ as the number of horizontal lozenges, which intersect the line
$x=x_0$ and which are below $(x_0,y_0)$, i.e.\ whose vertical coordinate is less than $y_0$.
\begin{proposition}[\cite{CKP}, \cite{KOS}]
\label{Proposition_limit_shape} Suppose that $\lambda(N)\in\widehat U(N)$, $N=1,2,\dots$, is a
sequence of
 signatures which satisfies a technical assumption of Definition \ref{definition_regularity}, and let $H_N(x,y)$ be the random height
  function of uniformly random tiling $\Upsilon_{\lambda(N)}$. Then
 as $N\to\infty$ for any $0<x<1$ and any $y\in\mathbb R$ the normalized height function $\frac{1}{N} H(Nx,Ny)$ converges to a \emph{deterministic} limit function,
 which can be found as a solution of a certain variational problem.
\end{proposition}

In order to link Proposition \ref{Proposition_limit_shape} and Theorem \ref{Theorem_main} observe
that for $M=1,\dots,N$ every tiling of the strip of Figure \ref{Fig_polyg_domain} has exactly $M$
horizontal lozenges at vertical line $x=M$. Coordinates of these lozenges can be encoded by a
signature from $\widehat U(M)$.

\begin{proposition}
\label{Proposition_tilings_and_restrictions} Take two integers $0<M<N$ and let $\lambda\in\widehat
U(N)$. Let $\mu\in\widehat U(M)$ be the \emph{random} signature which encodes the positions of
horizontal lozenges on vertical line $x=M$ in uniformly random lozenge tiling $\Upsilon_\lambda$.
Then the distribution of $\mu$ is given by the measure $\rho^\pi$ of \eqref{eq_rep_measure} , where
 $\pi$ is the restriction of irreducible representation $\pi^\lambda$ of $U(N)$ to the
subgroup $U(M)\subset U(N)$.
\end{proposition}
\begin{proof}
 This is a reformulation of the well--known \emph{branching rule} for the restrictions of
 representations of $U(N)$, see \cite{Zh}, \cite{FH}. The statement is also explained in
 \cite{BK}, \cite{GP}, \cite{BP_lectures}.
\end{proof}

Proposition \ref{Proposition_tilings_and_restrictions} yields, in particular, that the height
function of Proposition \ref{Proposition_limit_shape} as a function of $y$ with fixed $x=\lfloor
\alpha N\rfloor$ is precisely the (scaled) distribution function of the measure
$\rho^{\alpha,\lambda(N)}$ of Theorem \ref{Theorem_main}. This gives a direct relation between the
convergence of height functions in Proposition \ref{Proposition_limit_shape} and convergence of
measures in Theorem \ref{Theorem_main}. The difference here is that the limit in Proposition
\ref{Proposition_limit_shape} is described as a solution to a certain (complicated) variational
problem, while the limit in Theorem \ref{Theorem_main} is given in \ref{Theorem_linearization}
through its generating function.

For example, the concentration theorem for the restrictions of irreducible representations  with
rectangular signatures $(\underbrace{\beta N,...\beta N}_{\gamma N},
\underbrace{0,...0}_{(1-\gamma)N})$ corresponds to the limit shape theorem for the lozenge tilings
of hexagons or, equivalently, for \emph{boxed plane partitions}. In this case the variational
problem admits an explicit solution and the formulas for the limit shape are known, see
\cite{CLP}, \cite{Gorin-FAA}. On the other hand, our Theorem \ref{Theorem_main} gives an
alternative derivation for the limit shape theorem and Theorem \ref{Theorem_linearization} links
it to the free projections.

\begin{figure}[h]
\begin{center}
 \noindent{\scalebox{0.7}{\includegraphics{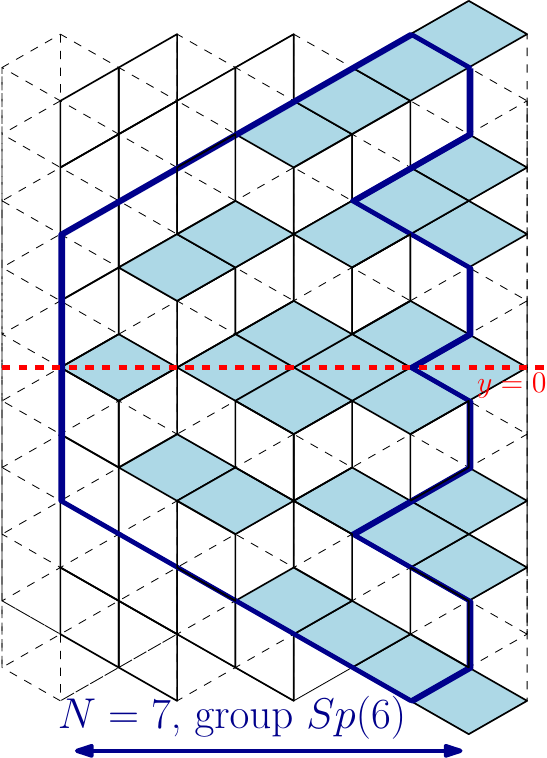}}} \hskip 1cm {\scalebox{0.7}{\includegraphics{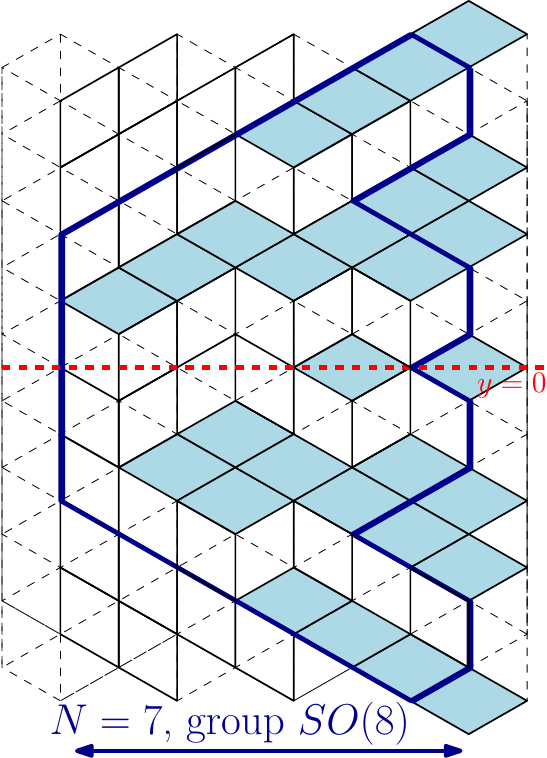}}}
 \caption{Symmetric lozenge tilings, corresponding to representations of symplectic group (left panel) and orthogonal group (right panel).
 \label{Fig_polyg_domain_sym}}
\end{center}
\end{figure}

\smallskip

The restrictions of representations of symplectic and orthogonal groups are related to the lozenge
tilings which are symmetric with respect to the line $y=0$, see Figure \ref{Fig_polyg_domain_sym}.
Put it otherwise, for each $M=1,2,\dots,N$ the coordinates of $M$ horizontal lozenges on the
vertical line $x=N$ should be symmetric around zero. When $M$ is even, there is a unique way to
prescribe what does it mean for a collection of $M$ numbers to be symmetric. However, when $M$ is
odd, there are two ways to do this: we have to specify what's happening with the middle number
(i.e.\ with the $\frac{M+1}{2}$th). We could either say that this number should be zero, as in the
left panel of Figure \ref{Fig_polyg_domain_sym}, or we could say that all the collection
\emph{except} for the middle number is symmetric (and there are no restrictions on this number) as
in the right panel of Figure \ref{Fig_polyg_domain_sym}. We call the former \emph{strongly
symmetric} tilings and the latter \emph{weakly symmetric} tilings.

Similarly to the identification of Figure \ref{Fig_polyg_domain}, we identify signature
$\lambda_1\ge\lambda_2\ge\dots\ge\lambda_N\ge 0$ labeling irreducible representation of $Sp(2N)$
with a strictly symmetric collection of $2N+1$ horizontal lozenges (the signature itself encodes
the positive coordinates of lozenges). We further identify signature
$\lambda_1\ge\lambda_2\ge\dots\ge\lambda_N\ge 0$ labeling irreducible representation of $SO(2N+1)$
with a weakly symmetric collection of $2N$ horizontal lozenges and signature
$\lambda_1\ge\lambda_2\ge\dots\ge\lambda_N$ labeling irreducible representation of $SO(2N)$ with a
weakly symmetric collection of $2N-1$ horizontal lozenges ($\lambda_N$, which might be negative,
corresponds to the middle $N$th lozenge).

\begin{proposition}
\label{Proposition_tilings_and_restrictions_BCD}
 In the same sense as in Proposition \ref{Proposition_tilings_and_restrictions}
the restrictions of representations of $Sp(2N)$ to $Sp(2M)\subset Sp(2N)$ ($0<M<N$) correspond to
the horizontal lozenges on the line $x=2M+1$ of uniformly random strongly symmetric tilings of
domain of width $2N+1$. The restrictions of representations of $SO(N)$ to $SO(M)\subset SO(N)$
($1<M<N$) correspond to the horizontal lozenges  on the line $x=M+1$ of uniformly random weakly
symmetric tilings of domain of width $N+1$.
\end{proposition}
\begin{proof} This is again a reformulation of the classical branching rules for the
representations of $Sp(2N)$ and $SO(N)$. There are various formulations of the branching rules,
the required form for the symplectic group is best explained in \cite{Kirillov} and for the
orthogonal groups the lozenge tilings interpretation equivalent to weakly symmetric tilings is
described in \cite{BK-O}.
\end{proof}

Now a combination of Propositions \ref{Proposition_tilings_and_restrictions},
\ref{Proposition_tilings_and_restrictions_BCD} with Theorem \ref{Theorem_linearization} gives the
following new statement in the spirit of Proposition \ref{Proposition_limit_shape}.

\begin{proposition}
\label{Proposition_limit_shape_sym} Suppose that $\lambda(N)\in\widehat U(N)$, $N>0$, is a
sequence of
 signatures which satisfies a technical assumption of Definition \ref{definition_regularity}, and also such that $\lambda(N)$ is strictly (weakly) symmetric.
 Let $H_N(x,y)$ be the random height
  function of uniformly random tiling $\Upsilon_{\lambda(N)}$ and let $\widetilde H_N(x,y)$ be the random height function of uniformly random weakly (strictly)
 symmetric
   tiling of the same domain.
  Then as $N\to\infty$ for any $0<x<1$ and any $y\in\mathbb R$ the normalized height functions $\frac{1}{N} H(Nx,Ny)$ and $\frac{1}{N} \widetilde H(Nx,Ny)$ converge
  to \emph{the same} deterministic limit function.
\end{proposition}
As before, the limit function can be described either in terms of the variational problem or
through the procedure of Theorem \ref{Theorem_linearization}.

\subsection{$U(\infty)$ and infinite divisibility}
\label{Section_U_infty} In the previous section we were discussing the law of large numbers (i.e.\
limit shape theorems) for the restrictions of representations of $U(N)$, $SO(N)$, $Sp(2N)$. One can
study similar problems with $N=\infty$, i.e.\ study restrictions of representations and characters
of \emph{infinite--dimensional} groups.

The most well-known example arising from such restrictions is the \emph{Plancherel measure} for the
symmetric groups $S(n)$, which describes the decomposition of \emph{biregular} representation of
the infinite symmetric group $S(\infty)$ into irreducible components. The law of large numbers for
the Plancherel measure is the celebrated Logan--Shepp--Vershik--Kerov limit shape theorem
\cite{LS}, \cite{VK_Plan}. Similar limit theorems exist for other representations of $S(\infty)$,
see \cite{GGK}.

Analogues of the Plancherel measure for groups $U(\infty)$ and $SO(\infty)$ and their asymptotics
were studied only much later by Borodin and Kuan \cite{BK}, \cite{BK-O}. In a recent article
\cite{BBO} Borodin, Olshanski and one of the authors obtained the limit shape theorem for the
restrictions of a rich family of characters (representations) of $U(\infty)$. In the spirit of
Theorem \ref{Theorem_main}, the limit shapes of \cite{BBO} can be interpreted as probability
measures on $\mathbb R$. We are especially interested in these measures because of their special
property: they are \emph{infinitely--divisible} with respect to the quantized free convolution. Let
us describe this property in more detail.

Consider four finite (not necessarily probability) measures $\mathfrak A^+$, $\mathfrak A^{-}$,
$\mathfrak B^+$, $\mathfrak B^{-}$ on $\mathbb R_{\ge 0}$ with compact supports and such that the
supports of $\mathfrak B^\pm$ are subsets of $[0,b^\pm]$, respectively, with $b^++b^-\le 1$. Set
$$
\mathcal M^{\mathfrak A^+}(z)=M_1(\mathfrak A^+)+M_2(\mathfrak A^+)z + M_2(\mathfrak
A^+)z^2+\dots,
$$
where $M_k(\mathfrak A^+)$ is the $k$th moment of the measure $\mathfrak A^+$, and similarly for
$\mathfrak A^{-}$, $\mathfrak B^+$, $\mathfrak B^{-}$. Take also two reals $\gamma^+,\gamma^-\ge
0$. Let $\mathfrak C$ denote the sextuple of parameters $(\mathfrak A^+,\mathfrak A^-,\mathfrak
B^+,\mathfrak B^-,\gamma^+,\gamma^-)$.

\begin{proposition}
\label{Prop_inf_divisible}
 There exists a unique probability measure $\mathcal M(\mathfrak C)$ on
 $\mathbb R$ such that its $R$--transform is given by:
\begin{multline*}
R^{quant}_{\mathcal M(\mathfrak C)}(u)=R_{\mathcal M(\mathfrak C)}(u)+
\frac{1}{u}-\frac{1}{1-e^{-u}} \\= e^u\gamma^+ -e^{-u}\gamma^- + e^u M^{\mathfrak B^+} (1-e^u) +
e^u \mathcal M^{\mathfrak A^+} (e^u-1) - e^{-u}\mathcal M^{\mathfrak B^-} (1- e^{-u}) -
e^{-u}\mathcal M^{\mathfrak A^-} (e^{-u} - 1).
\end{multline*}
Moreover, $\mathcal M(\mathfrak C)$ is infinitely--divisible with respect to the quantized
free-convolution, i.e.\ for each $n=1,2,\dots$ there exists a probability measure $\mathcal M_n$
such that
$$
 \mathcal M(\mathfrak C)=\underbrace{\mathcal M_n\otimes\dots\otimes
 \mathcal M_n}_{n\text{ factors}}.
$$
\end{proposition}
\begin{proof}
 The existence of the measure is proved in \cite{BBO}, see Theorem 3.2 and equations (3.3)-(3.5)
 there (note that any measure can be approximated by discrete ones). The
 infinite divisibility is immediate, since we can choose
 $\mathcal M_n=\mathcal M(\mathfrak C/n)$ corresponding to the sextuple
 $\mathfrak C/n=(\mathfrak A^+/n,\mathfrak A^-/n,\mathfrak B^+/n,\mathfrak
 B^-/n,\gamma^+/n,\gamma^-/n)$.
 \end{proof}

We close this section by the remark that for the (conventional) convolution the classification of
all infinitely--divisible measures is given by the classical Levy-Khintchine formula and for the
free convolution the classification theorem was proved in \cite{Vo2}, \cite{BV}. The list of
Proposition \ref{Prop_inf_divisible} does not exhaust the class of infinitely--divisible measures
with respect to the quantized free convolution (there is an example of an infinitely--divisible
measure outside this list) and it would be interesting to complete the classification. We believe
that the measures of Proposition \ref{Prop_inf_divisible} should play the role of (Free) Compound
Poisson distributions in this classification.

\section{Characters, differential operators, and  asymptotics}
\label{Section_characters} Our approach to the study of the asymptotic decompositions of
irreducible representations of $G(N)$ as $N\to\infty$ is based on the knowledge of the asymptotics
of  normalized logarithms of their characters. We will also employ certain differential operators
which we present in this section.

\subsection{Characters of irreducible representations}
Let $\chi^{G(N)}_\lambda$ denote the character of the irreducible representation of $G(N)$
parameterized by $\lambda\in\hG(N)$. We identify $\chi^{G(N)}_\lambda$ with a symmetric Laurent
polynomial $\chi^{G(N)}_\lambda(u_1,\dots,u_N)$. For $G(N)=U(N)$, $u_i$'s stand for the eigenvalues
of a unitary matrix. When $G(N)\ne U(N)$, the eigenvalues of an element $G(N)$ form pairs $z_1,
z_1^{-1}; z_2,z_2^{-1}$,\dots and we choose $u_i$ to be one element from $i$th pair (the characters
are invariant under $u_i\to u_i^{-1}$, so it does not matter, which one we choose). We have (see
e.g.\ \cite[Section 24.2]{FH}):
\begin{equation}
\label{eq_character_A}
 \chi^{U(N)}_\lambda(u_1,\dots,u_N)=\frac{\det\left[u_i^{\lambda_j+N-j}\right]}{\det\left[u_i^{N-j}\right]}= \frac{\det\left[u_i^{\lambda_j+N-j}\right]}{\prod\limits_{i<j}
 (u_i-u_j)},
\end{equation}
\begin{equation}
\label{eq_character_B}
 \chi^{SO(2N+1)}_\lambda(u_1,\dots,u_N)=\frac{\det\left[u_i^{\lambda_j+N+1/2-j}-u_i^{-(\lambda_j+N+1/2-j)}\right]} {\prod\limits_{i=1}^N
 (u_i^{1/2}-u_i^{-1/2}) \prod\limits_{i<j} (u_i+u_i^{-1}- (u_j+u_j^{-1}))},
\end{equation}
\begin{equation}
\label{eq_character_C}
 \chi^{Sp(2N)}_\lambda(u_1,\dots,u_N)=\frac{\det\left[u_i^{\lambda_j+N+1-j}-u_i^{-(\lambda_j+N+1-j)}\right]} {\prod\limits_{i=1}^N
 (u_i-u_i^{-1}) \prod\limits_{i<j} (u_i+u_i^{-1}- (u_j+u_j^{-1}))},
\end{equation}
\begin{equation}
\label{eq_character_D}
 \chi^{SO(2N)}_\lambda(u_1,\dots,u_N)=\frac{\det\left[u_i^{\lambda_j+N-j}+u_i^{-(\lambda_j+N-j)}\right]+\det\left[u_i^{\lambda_j+N-j}-u_i^{-(\lambda_j+N-j)}\right]}{\prod\limits_{i<j}(u_i+u_i^{-1}-(u_j+u_j^{-1}))}.
\end{equation}

\subsection{Asymptotic expansions of characters}

Recall that $R_\mes(z)$ is the Voiculescu $R$-transform, which was defined in Section
\ref{Section_results}. Integrating $R_\mes(z)$ termwise, set
\begin{equation}
\label{eq_H_function} \H_\mes(u)=\int_0^{\ln(u)}R_\mes(t)dt+\ln\left(\frac{\ln(u)}{u-1}\right),
\end{equation}
which should be understood as a power series in $(u-1)$.

\begin{lemma} If $\mes$ is a measure with compact support, then $\H_\mes(u)$ as a power series in $(u-1)$ is uniformly convergent in an open
neighborhood of $1$.
\end{lemma}
\begin{proof} Immediately follows from the definitions. \end{proof}

\begin{theorem}
\label{Theorem_character_asymptotics}
 Suppose that $\lambda(N)\in \hG(N)$, $N=1,2,\dots$ is a regular sequence of signatures, such that
 $$
  \lim_{N\to\infty} m^G[\lambda(N)]=\mes.
 $$
 Then for any
 $k=1,2,\dots$ we have
 \begin{equation}
 \label{eq_limit_of_logarithm}
  \lim_{N\to\infty} \frac{1}{\hat N}
  \ln\left(\frac{\chi^{G(N)}_{\lambda(N)}(u_1,\dots,u_k,1^{N-k})}{\chi^{G(N)}_{\lambda(N)}(1^N)}\right)=
  \H_\mes(u_1)+\dots+ \H_\mes(u_k),
 \end{equation}
 where the convergence is uniform over an open (complex) neighborhood of $(1,\dots,1)$, $\hat N=N$
 for the unitary group and $\hat N=2N$ for the symplectic and orthogonal groups.
\end{theorem}
\noindent {\bf Remark. } $1^M$ here and below means the sequence of $M$ ones
$(\underbrace{1,\dots,1}_{M})$.

Pointwise identity \eqref{eq_limit_of_logarithm} for the unitary groups and real $u_i$  first
appeared in \cite{GM}. In \cite{GP} it was extended to complex $u_i$ and other classical Lie
groups. Note that neither of the papers contain the statement about uniformity. However, the
techniques of \cite{GP} readily imply this uniformity and we fill in all the details in the
Appendix.

\subsection{Differential operators}

Let $V^{G(N)}(u_1,\dots,u_N)$ denote the denominator in formulas
\eqref{eq_character_A}--\eqref{eq_character_D}. In particular, $V^{U(N)}=\prod_{i<j}(u_i-u_j)$.
Introduce a differential operator acting on symmetric functions in variables $u_1,\dots,u_N$:
\begin{equation}
\label{eq_def_dif}
 \D^{G(N)}_k=\frac{1}{V^{G(N)}} \circ \left(\sum_{i=1}^N \left(u_i \frac{\partial}{\partial
 u_i}\right)^k \right) \circ V^{G(N)},
\end{equation}
where $V^{G(N)}$ in the last formula is understood as an operator of multiplication by $V^{G(N)}$.

\begin{proposition}
\label{Proposition_eigenfunctions} The characters $\chi^{G(N)}_\lambda(u_1,\dots,u_N)$,
$\lambda\in\hG(N)$ are eigenfunctions of $\D^{G(N)}_k$ for all $k=0,1,\dots$ if $G(N)=U(N)$ and for
even $k=0,2,4,\dots$ if $G(N)\ne U(N)$. The corresponding eigenvalues are
$$
 \D^{G(N)}_k \chi^{G(N)}_\lambda(u_1,\dots,u_N) = \sum_{i=1}^N (\mu_i)^k
 \chi^{G(N)}_\lambda(u_1,\dots,u_N),
$$
where (depending on the group $G(N)$) $\mu_i=\lambda_i+N-i$ for $U(N)$ and $SO(2N)$,
$\mu_i=\lambda_i+N-i+1/2$ for $SO(2N+1)$ and $\mu_i=\lambda_i+(N+1)-i$ for $Sp(2N)$.
\end{proposition}
\noindent {\bf Remark.} In our limit regime $\lambda_i$ grow linearly as $N\to\infty$, thus, the
difference between the definitions of $\mu_i$ for different groups becomes negligible.
\begin{proof}[Proof of Proposition \ref{Proposition_eigenfunctions}]
Immediate from Weyl characters formulas \eqref{eq_character_A}--\eqref{eq_character_D}.
\end{proof}

We also need another family of differential operators in the study of Perelomov--Popov measure,
which are defined as follows.

For $U(N)$ set

$$
 \D^{PP,U(N)}_k=\frac{1}{V^{U(N)}} \circ \left(\sum_{i=1}^N \frac{\partial}{\partial u_i} \left(u_i \frac{\partial}{\partial
 u_i}\right)^{k-1} \right) \circ V^{U(N)}.
$$
For other series ($G(N)\ne U(N)$) set
\begin{equation}
\label{eq_PP_op_BCD}
 \D^{PP,G(N)}_{2k}=\frac{1}{V^{G(N)}} \circ \left(\sum_{i=1}^N (u_i+{u_i}^{-1}) \left(u_i \frac{\partial}{\partial
 u_i}\right)^{2k} \right) \circ V^{G(N)},
\end{equation} and
\begin{equation}
\label{eq_PP_op_BCD_2}
 \D^{PP,G(N)}_{2k+1}=\frac{1}{V^{G(N)}} \circ \left(\sum_{i=1}^N ({u_i}^{-1}-u_i) \left(u_i \frac{\partial}{\partial
 u_i}\right)^{2k+1} \right) \circ V^{G(N)}.
\end{equation}

Let us now explain the interplay between differential operators and moments of random measures
$m[\rho]$.

Let $\rho$ be a probability measure on $\hG(N)$.
\begin{definition}
\label{Def_character_generating} A \emph{character generating function}
$\CG^{G(N)}_\rho(u_1,\dots,u_N)$ is a symmetric Laurent power series  in $(u_1,\dots,u_N)$ given by
$$
 \CG^{G(N)}_\rho(u_1,\dots,u_N)=\sum_{\lambda\in\hG(N)} \rho(\lambda)
 \frac{\chi^{G(N)}_\lambda(u_1,\dots,u_N)}{\chi^{G(N)}_\lambda(1^N)}.
$$
\end{definition}
In what follows we always assume that the measure $\rho$ is such that this (in principle, formal)
sum is uniformly convergent in an open neighborhood of $(1,\dots,1)$. Note that we always have
$\CG^{G(N)}_\rho(1,\dots,1)=1$. In all our examples $\rho$ is such that the sum in Definition
\ref{Def_character_generating} is, actually, finite.

\begin{proposition}
\label{proposition_moments_and_operators_A} Let $\rho$ be a probability measure on $\widehat U(N)$
whose character generating function is well-defined in an open neighborhood of $(1,\dots,1)$. Then
 for $k=1,2,\dots$ the following formula  for the expectations of moments of random
measures $m^{A}[\rho]$ and $m^{A}_{PP}[\rho]$ holds:
\begin{equation}
\label{eq_moments_uniform}
 \mathbb E \left( \int_{\mathbb R} x^k m^A[\rho](dx) \right)^m = \frac{1}{N^{m(k+1)}} \left(\D^{U(N)}_k\right)^m
 \CG^{U(N)}_\rho(u_1,\dots,u_N)\Bigl|_{u_1=\dots=u_N=1},
\end{equation}
\begin{equation}
\label{eq_moments_PP}
 \mathbb E \left( \int_{\mathbb R} x^k m^A_{PP}[\rho](dx) \right) = \frac{1}{N^{m(k+1)}} \left(\D^{U(N),PP}_k\right)
 \CG^{U(N)}_\rho(u_1,\dots,u_N)\Bigl|_{u_1=\dots=u_N=1}.
\end{equation}
\end{proposition}
\begin{proof}
 For the counting measures we have
 $$
 \mathbb E \left( \int_{\mathbb R} x^k m^A[\rho](dx) \right)^m=\sum_{\lambda\in\hG(N)} \rho(\lambda)
 \left( \frac{1}{N} \sum_{i=1}^N \left(\frac{\lambda_i+N-i}{N}\right)^k \right)^m.
 $$
 On the other hand, expanding $\CG^{U(N)}_\rho(u_1,\dots,u_N)$ into the sum of characters and applying
 $\left(\D^{U(N)}_k\right)^m$ using Proposition \ref{Proposition_eigenfunctions} we arrive at the
 same expression, which proves \eqref{eq_moments_uniform}. For the Perelomov--Popov measures note that
$$
 \D^{U(N),PP}_k \chi^{U(N)}_\lambda(u_1,\dots,u_N) = \sum_{i=1}^N (\mu_i)^k
 \chi^{U(N)}_{\lambda^{(i-)}}(u_1,\dots,u_N),
$$
and use the same argument.
\end{proof}

To state an analogue of Proposition \ref{proposition_moments_and_operators_A} for root systems
$B,C,D$ it is convenient to slightly redefine the measures corresponding to the signatures as
follows:
$$
\widehat m^{B}[\lambda] = \frac{1}{2N}  \sum_{i=1}^{N} \left( \delta\left(\frac{ \lambda_i
+N-i+1/2}{2N}\right) + \delta\left(\frac{ i-1/2-\lambda_i-N}{2N}\right) \right),
$$
$$
\widehat m^{C}[\lambda] = \frac{1}{2N} \sum_{i=1}^{N} \left(  \delta\left(\frac{ \lambda_i + N+1 -
i}{2N}\right) +  \delta\left(\frac{i-1 -\lambda_{i}-N}{2N}\right) \right),
$$
$$
\widehat m^{D}[\lambda] = \frac{1}{2N} \sum_{i=1}^{N} \left(  \delta\left(\frac{ \lambda_i
+N-i}{2N}\right) + \delta\left(\frac{i -\lambda_i-N}{2N}\right) \right).
$$
Note that when $N$ is large the above measures (up to a small error) differ from the measures $
m^{B,C,D}[\lambda]$ by a shift by $1/2$. On the other hand, the advantage of the measures $\widehat
m^{B,C,D}[\lambda]$ is that they are \emph{symmetric} with respect to the origin, thus, to study
their asymptotics it is enough to consider only even moments. For the latter we have:

\begin{proposition}
\label{proposition_moments_and_operators_BCD} Let $\rho$ be a probability measure on $\widehat
G(N)$, $G(N)\ne U(N)$ whose character generating function is well-defined in an open neighborhood
of $(1,\dots,1)$, then
 for $k=1,2,\dots$ the following formula holds for the expectations of even moments of random
measures $\widehat m^{G}[\rho]$:
\begin{equation}
\label{eq_moments_uniform_BCD}
 \mathbb E \left( \int_{\mathbb R} x^{2k} \widehat m^G[\rho](dx) \right)^m = \frac{1}{2^{2mk} N^{m(2k+1)}} \left(\D^{G(N)}_{2k}\right)^m
 \CG^{G(N)}_\rho(u_1,\dots,u_N)\Bigl|_{u_1=\dots=u_N=1}.
\end{equation}
\end{proposition}
\begin{proof} Same argument as in Proposition \ref{proposition_moments_and_operators_A}.
\end{proof}

Finally, the moments of the Perelomov--Popov measures $m^G_{PP}[\lambda]$ for root systems $B$, $C$
and $D$ can be extracted using the operators \eqref{eq_PP_op_BCD}, \eqref{eq_PP_op_BCD_2}. We leave
the exact statement to an interested reader.

\section{Asymptotics of random measures}
\label{Section_asymptotics_of_measures}

In this section we explain how the knowledge of the asymptotics of the logarithms of characters can
be used to establish asymptotic results for various measures related to these characters.

\smallskip

Let $\rho(N)$ be a sequence of measures such that for each $N=1,2,\dots$, $\rho(N)$ is a
probability measure on $\hG(N)$.
\begin{theorem}
\label{theorem_moment_convergence}
 Suppose that $\rho(N)$ is such that for every $k$
$$
 \lim_{N\to\infty} \frac{1}{N} \ln( \CG^{U(N)}_{\rho(N)}(u_1,\dots,u_k, 1^{N-k}) ) = Q(u_1)+\dots+Q(u_k),
$$
where $Q$ is an analytic function in a neighborhood of $1$ and the convergence is uniform in an
open (complex) neighborhood of $(1,\dots,1)$. Then random measures $m^A[\rho(N)]$ converge as
$N\to\infty$ in probability, in the sense of moments to a \emph{deterministic} measure $\mes$ on
$\mathbb R$, whose moments are given  by
\begin{equation}
\label{eq_limit_moments_A}
 \int_{\mathbb R} x^k \mes(dx)= \sum_{\ell=0}^k \frac{k!}{\ell! (\ell+1)! (k-\ell)!} \frac{\partial^\ell}{\partial u^\ell}\left(u^k
 Q'(u)^{k-l}\right)\Biggr|_{u=1}.
\end{equation}
\end{theorem}

\begin{theorem}
\label{theorem_moment_convergence_BCD}
 Suppose that for $G(N)\ne U(N)$, $\rho(N)$ is such that for every $k$
$$
 \lim_{N\to\infty} \frac{1}{N} \ln( \CG^{G(N)}_{\rho(N)}(u_1,\dots,u_k, 1^{N-k}) ) = Q(u_1)+\dots+Q(u_k),
$$
where $Q$ is an analytic function in a neighborhood of $1$ and the convergence is uniform in an
open (complex) neighborhood of $(1,\dots,1)$. Then random measures $\widehat m^G[\rho(N)]$ converge
as $N\to\infty$ in probability, in the sense of moments to a \emph{deterministic} measure $\mes$ on
$\mathbb R$, whose odd moments are zero, while even moments are given  by
\begin{equation}
\label{eq_limit_moments_BCD}
 \int_{\mathbb R} x^{2k} \mes(dx)= 2^{-2k} \sum_{\ell=0}^{2k} \frac{(2k)!}{\ell! (\ell+1)! (2k-\ell)!}
 \frac{\partial^\ell}{\partial z^{\ell}}\left((z^2-1)^k
 \widehat Q'(z)^{2k-l}\right)\Biggr|_{z=1},
\end{equation}
where $\widehat Q(z)$ is defined through
$$
 \widehat Q\left(\frac{u+u^{-1}}2\right)= Q(u).
$$
\end{theorem}

\noindent {\bf Remark 1.} Theorem \ref{theorem_moment_convergence} is inspired by the results of the paper
\cite{BBO}, which was in
preparation when this project started. In particular, Theorem \ref{theorem_moment_convergence} can be used to get an
alternative proof of the limit shape theorem for the
decompositions of restrictions of the characters of $U(\infty)$, cf.\ \cite[Theorem
3.2]{BBO}. Our techniques are different from
those
of \cite{BBO}: the latter used differential operators on the group, while we use differential operators on the
eigenvalues. One advantage of our approach is that it is generalized to symplectic and orthogonal groups with
relatively small modifications; on the other hand, as far as the authors know, the group approach is not yet developed
in this direction.

\noindent{\bf Remark 2.} Note that characters for root systems $B,C$ and for root systems $D$ when
at least one of the variables is set to $1$, are polynomials in $u_i+u_i^{-1}$. This guarantees
that an analytic $\widehat Q$ in Theorem \ref{theorem_moment_convergence_BCD} exists.

\noindent{\bf Remark 3.} The key part in Theorems \ref{theorem_moment_convergence},
\ref{theorem_moment_convergence_BCD} is that the moments of the limit measures $\mes$ are uniquely
defined by $Q(u)$; the exact form of this dependence is less important.

\noindent{\bf Remark 4.} Note that in Theorem \ref{theorem_moment_convergence_BCD} we deal with
measures $\widehat m^G[\rho(N)]$. However, since asymptotically as $N\to\infty$ they differ from $
m^G[\rho(N)]$ by the deterministic shift by $1/2$, the convergence to deterministic limits holds
also for $m^G[\rho(N)]$.

\smallskip

In the proof of Theorems \ref{theorem_moment_convergence},
\ref{theorem_moment_convergence_BCD} we will use the operators $\D^{G(N)}_{k}$.
Although, one can study the Perelomov--Popov measures using the operators
$\D^{G(N),PP}_{k}$, we will use another (simpler) way to access them relying on the
following statement.
\begin{theorem}
\label{theorem_moment_link_to_PP}
 Suppose that $\rho(N)$ is such that random measures $m^G[\rho(N)]$ converge as
$N\to\infty$ in probability, in the sense of moments to a \emph{deterministic} measure $\mes$ on
$\mathbb R$. Then random measures $m^G_{PP}[\rho(N)]$ also converge as $N\to\infty$ in probability,
in the sense of moments to a \emph{deterministic} measure $\mes_{PP}$ on $\mathbb R$. Moreover, if
we set
$$
 s_k=\int_{\mathbb R} x^k \mes(dx),\quad c_k=\int_{\mathbb R} x^k \mes_{PP}(dx), \quad k=0,1,2,\dots
$$
then
$$1 - \sum_{k=0}^{\infty} c_k z^{k+1} = \exp \left(- \sum_{k=0}^{\infty} s_k z^{k+1} \right).$$
\end{theorem}

Note that this is Theorem \ref{theorem_moment_link_to_PP} which serves as a motivation for the
definition of the map $\rho\mapsto Q(\rho)$ of Theorem \ref{theorem_intertwining}. In the rest of
this section we prove the above three theorems and also Theorem \ref{theorem_intertwining}.

\subsection{Two lemmas}

The following technical lemmas will be crucial for our analysis.

\begin{lemma}
\label{Lemma_Vandermonde_sum} Take $n>0$, let $I^{(n)}$ be the set of all pairs $1\le a<b\le n$ and
suppose that $P\subset I^{(n)}$. Let $f(z_1,\dots,z_n)$ be an analytic function in a neighborhood
of $(1,\dots,1)$, and set
$$
 f_P(z_1,\dots,z_n)=Sym\left(\frac{f(z_1,\dots,z_n)}{\prod_{(a,b)\in P}
 (z_a-z_b)}\right)=\frac{1}{n!}\sum_{\sigma\in S(n)} \frac{f(z_{\sigma(1)},\dots,z_{\sigma(n)})}{\prod_{(a,b)\in P}
 (z_{\sigma(a)}-z_{\sigma(b)})}.
$$
Then $f_P$ is also an analytic function in a (perhaps, smaller) neighborhood of $(1,\dots,1)$.
Further, if $f^t(z_1,\dots,z_n)$, $t=1,2,\dots$ is a sequence of analytic functions converging to
$0$ uniformly in a neighborhood of $(1,\dots,1)$, then so is the sequence $f_P^t(z_1,\dots,z_n)$.
\end{lemma}
\begin{proof} We will shift the variables $z_i=1+x_i$ and argue in terms of $x_i$.
 First, suppose that $f$ is a monomial, $f=x_1^{k_1}\cdots x_n^{k_n}$, then we have
\begin{multline*}
 f_P(x_1,\dots,x_n) \prod_{i<j}(x_i-x_j)= \frac{1}{n!}\sum_{\sigma\in S(n)}
 (-1)^{\sigma}\frac{(x_{\sigma(1)}^{k_1},\dots,x_{\sigma(n)}^{k_n})\prod_{i<j}(x_{\sigma(i)}-x_{\sigma(j)})}{\prod_{(a,b)\in P}
 (x_{\sigma(a)}-x_{\sigma(b)})}
 \\= \frac{1}{n!}\sum_{\sigma\in S(n)} (-1)^{\sigma} x_{\sigma(1)}^{k_1}\cdots x_{\sigma(n)}^{k_n} {\prod_{(a,b)\in  I^{(n)} \setminus P}
 (x_{\sigma(a)}-x_{\sigma(b)})},
\end{multline*}
where $(-1)^{\sigma}$ is $\pm 1$ depending on whether a permutation $\sigma$ is even or odd. This
expansion shows that $f_P(x_1,\dots,x_n) \prod_{i<j}(x_i-x_j)$ is a skew--symmetric polynomial in
$x_1,\dots,x_n$, thus, it is divisible by $\prod_{i<j}(x_i-x_j)$ and $f_P(x_1,\dots,x_n)$ is a
symmetric polynomial in $x_1,\dots,x_n$ (in particular, it is analytic in any neighborhood of
$(0,\dots,0)$).

Further, observe that $n! f_P(x_1,\dots,x_N) \prod_{i<j}(x_i-x_j)$ is a (signed) sum of at most
$n!$ elementary skew-symmetric polynomials
\begin{equation}
\label{eq_alt_sum}
 \sum_{\sigma\in S(n)} (-1)^{\sigma} x_{\sigma(1)}^{m_1}\cdots x_{\sigma(n)}^{m_n},
\end{equation}
and $|m_i-k_i|\le n$. When we divide the alternating sum \eqref{eq_alt_sum} by
$\prod_{i<j}(x_i-x_j)$ we arrive at Schur polynomial $s_\lambda(x_1,\dots,x_n)$. (Here
$m_i=\lambda_i+n-i$.). Now we can expand Schur polynomials into monomials using the
\emph{combinatorial formula} (see e.g.\ \cite[Chapter I]{M}) for them. The total degree for each
monomial in this expansion is $m_1+\dots+m_n$, the coefficients are non-negative integers, and the
weighted number of terms is $s_\lambda(1^n)$, which simplifies to a polynomial in $\lambda_i$,
using \eqref{eq_Weyl_dimension}. We conclude that for $f=x_1^{k_1}\cdots x_n^{k_n}$, we have
\begin{equation}
\label{eq_P_transform}
 f_P=\sum_{p_1\ge0,\dots,p_n\ge 0} d^{k_1,\dots,k_n}_{p_1,\dots,p_n} x_1^{p_1}\cdots x_n^{p_n},
\end{equation}
 where coefficients $d^{k_1,\dots,k_n}_{p_1,\dots,p_n}$ vanish unless $|\sum_i k_i-\sum_i
p_i|\le n^2$ and
\begin{equation}
\label{eq_estimate}
 \sum_{p_1\ge0,\dots,p_n\ge 0} |d^{k_1,\dots,k_n}_{p_1,\dots,p_n}| \le g(k_1,\dots,k_n),
\end{equation}
where $g$ is a certain polynomial in $k_1,\dots,k_n$.

Now let $f$ be an analytic function in the neighborhood of $(0,\dots,0)$, i.e.
\begin{equation}
\label{eq_x1}
 f=\sum_{k_1\ge 0,\dots, k_n\ge 0} c_{k_1,\dots,k_n} x_1^{k_1}\cdots x_n^{k_n}.
\end{equation}
The convergence of \eqref{eq_x1} implies that for some $R>0$,
\begin{equation}
\label{eq_x2} \sum_{k_1\ge 0,\dots, k_n\ge 0} c_{k_1,\dots,k_n} R^{k_1+\dots+k_n}<\infty.
\end{equation}
Plugging \eqref{eq_P_transform} into \eqref{eq_x1} we get the expansion
\begin{equation}
\label{eq_x1_P}
 f_P=\sum_{k_1\ge 0,\dots, k_n\ge 0} c^P_{k_1,\dots,k_n} x_1^{k_1}\cdots x_n^{k_n}.
\end{equation}
But now estimate \eqref{eq_estimate} together with \eqref{eq_x2} yields that for any $0<\eps<R$,
$$
 \sum_{k_1\ge 0,\dots, k_n\ge 0} |c^P_{k_1,\dots,k_n}| (R-\eps)^{k_1+\dots+k_n}<\infty.
$$
Hence, $f_P$ is analytic in a neighborhood of $(0,\dots,0)$.

Further, if $f^t$ is a sequence of functions converging to $0$ then for some $R>0$ the sums as in
\eqref{eq_x2} converge to $0$ as $n\to\infty$. We again conclude using \eqref{eq_estimate} that
similar sums for $f_P^t$, as in \eqref{eq_x1_P}, also converge to $0$ (perhaps, for smaller $R$)
and, thus $f_P^t$ uniformly converges to $0$ in a neighborhood of $(0,\dots,0)$.
\end{proof}

We also need the value of $f_P$ for one particular choice of $f$ and $P$.
\begin{lemma}
 Take $n>0$ and a function $g(z)$ analytic in a neighborhood of $1$. Then
\begin{multline}
\label{eq_x3} \lim_{z_i\to 1} \Biggl( \frac{g(z_1)}{(z_1-z_2)(z_1-z_3)\cdots (z_1-z_n)}
+\frac{g(z_2)}{(z_2-z_1)(z_2-z_3)\cdots (z_2-z_n)}\\
+\cdots+\frac{g(z_n)}{(z_n-z_1)(z_n-z_3)\cdots (z_n-z_{n-1})}\Biggr)=
\frac{\partial^{n-1}}{\partial z^{n-1}}\left(\frac{g(z)}{(n-1)!}\right)\Biggr|_{z=1}.
\end{multline}
\end{lemma}
\begin{proof}
 The proof of Lemma \ref{Lemma_Vandermonde_sum} shows that the sum under
 the limit in \eqref{eq_x3} is analytic in $x_i$. Therefore, it is continuous near
 the point $(1,\dots,1)$ and we can approximate this point from any direction. Let
 $z_i=1+\eps(i-1)$ with $\eps\to 0$. Expanding $g(z)$ in Taylor series, we get
\begin{multline}
\label{eq_x4} \eps^{1-n}\sum_{i=1}^n (-1)^{n-i}\frac{g(1)+\eps (i-1) g{'}(1)+\frac{\eps^2
(i-1)^2}{2!} g{''}(1)+\dots}{(i-1)!(n-i)!}
\\=\frac{\eps^{1-n} g(1)}{(n-1)!} \sum_{j=0}^{n-1} j^{0} {n-1 \choose j}(-1)^{n-1-j}+ \frac{\eps^{2-n} g'(1)}{(n-1)!2!} \sum_{j=0}^{n-1} j^1 {n-1 \choose
j}(-1)^{n-1-j}\\ +\dots+ \frac{g^{(n-1)}(1)}{(n-1)!(n-1)!} \sum_{j=0}^{n-1} j^{n-1} {n-1 \choose
j}(-1)^{n-1-j}+ o(\eps)
\end{multline}
Differentiating $k$ times expression $(1-z)^{n-1}$ with $k=0,\dots,n-2$ and substituting $z=1$ one
proves that for any polynomial $h$ of degree at most $n-2$,
\begin{equation}
\label{eq_x40} \sum_{j=0}^{n-1} h(j) {n-1 \choose j}(-1)^{n-1-j} =0.
\end{equation}
Moreover, using \eqref{eq_x40} we also get
$$
  \sum_{j=0}^{n-1} j^{n-1} {n-1 \choose
j}(-1)^{n-1-j}= \sum_{j=0}^{n-1} j(j-1)\cdots(j-n+2) {n-1 \choose j}(-1)^{n-1-j}= (n-1)!.
$$
Therefore, \eqref{eq_x4} transforms into
$$
 \frac{g^{(n-1)}(1)}{(n-1)!} + o(\eps). \qedhere
$$
\end{proof}

\subsection{Proof of Theorem \ref{theorem_moment_convergence}}

 First, write
\begin{equation}
\label{eq_x41} \CG^{G(N)}_{\rho(N)}(u_1,\dots,u_N) = \exp\left(\sum_{i=1}^N N Q(u_i)\right)
T_N(u_1,\dots,u_N).
\end{equation}
Since $\CG^{G(N)}_{\rho(N)}(1^N)=1$, the definition of $Q$ implies that $Q(1)=0$,
$T_N(1,\dots,1)=1$ and
\begin{equation}
\label{eq_x5}
 \lim_{N\to\infty} \frac{1}{N} \ln(T_N(u_1,\dots,u_k,1^{N-k}))=0
\end{equation}
for any $k$ and uniformly over an open neighborhood of $(1^k)$. Since \eqref{eq_x5} involves
uniformly converging analytic functions, we can differentiate it. The outcome is that each partial
derivative (of arbitrary order) of $T_N(u_1,\dots,u_k,1^{N-k})$ divided by
$NT_N(u_1,\dots,u_k,1^{N-k})$ tends to zero uniformly in a certain neighborhood of $(1^k)$.

We want to use Proposition \ref{proposition_moments_and_operators_A} to obtain the asymptotics of
the moments of $m^A[\lambda(N)]$. The formula \eqref{eq_moments_uniform} can be alternatively
written as
$$
 \mathbb E \left( \int_{\mathbb R} x^k m^A[\rho(N)](dx) \right)^m = \frac{1}{N^{m(k+1)}} \lim_{u_1,\dots,u_N\to 1}
  \dfrac{\left(\D^{U(N)}_k\right)^m
 \CG^{U(N)}_{\rho(N)}(u_1,\dots,u_N)}{\CG^{U(N)}_{\rho(N)}(u_1,\dots,u_N)}.
$$

Using  \eqref{eq_def_dif}, \eqref{eq_x41} and the Leibnitz rule we can write
$\left(\D^{U(N)}_k\right)^m
 \CG^{G(N)}_{\rho(N)}(u_1,\dots,u_N)$ as a huge linear combination of the terms of the following kind:
\begin{multline}
\label{eq_x6}
 \left(u_{g_1}\cdots u_{g_\gamma}\right)\cdot\left( \frac{\frac{\partial}{\partial u_{t_1}} \cdots \frac{\partial}{\partial
 u_{t_\tau}} \prod_{i<j}(u_i-u_j)}{\prod_{i<j}(u_i-u_j)}\right)\\ \times \left( \frac{\partial}{\partial u_{a_1}}  \cdots \frac{\partial}{\partial
 u_{a_\alpha}} \exp\left(\sum_{i=1}^N N Q(u_i)\right)\right) \cdot\left( \frac{\partial}{\partial u_{b_1}} \cdots \frac{\partial}{\partial
 u_{b_\beta}} T_N(u_1,\dots,u_N)\right),
\end{multline}
where $\gamma\le mk$ and $\alpha+\beta+\tau\le mk$. We can further expand the second factor in
\eqref{eq_x6} and get the terms
\begin{multline}
\label{eq_x33}
 \left(u_{g_1}\cdots u_{g_\gamma}\right)\cdot\left( \prod_{(a,b)\in P} \frac{1}{u_a-u_b}\right)\\ \times \left( \frac{\partial}{\partial u_{a_1}}  \cdots \frac{\partial}{\partial
 u_{a_\alpha}} \exp\left(\sum_{i=1}^N N Q(u_i)\right)\right) \cdot\left( \frac{\partial}{\partial u_{b_1}} \cdots \frac{\partial}{\partial
 u_{b_\beta}} T_N(u_1,\dots,u_N)\right),
\end{multline}
where $P\subset I^{(N)}$ and $I^{(N)}=\{(a,b)\mid 1\le a<b\le N\}$. The symmetry of the operator
$\left(\D^{G(N)}_k\right)^m$ implies that together with each term of the kind \eqref{eq_x33} all
the terms obtained by permuting the variables $u_i$, $i=1,\dots,N$ are also present. Let us call
\emph{the support} of the term \eqref{eq_x33} the union of the sets $\{{g_1},\dots g_{\gamma}\}$,
$\{b_1,\dots,b_\beta\}$, $\{a_1,\dots,a_\alpha\}$ and projections of $P$ on the first and second
coordinates. Further, two terms of the kind \eqref{eq_x33} are said to be of the same
\emph{combinatorial type} if one of them can be obtained from another by permuting the variables
$u_i$ and, perhaps, sign change. Note that for fixed $m$ and $k$, the set of possible combinatorial
types do not depend on $N$ as long as $N$ is large enough (as compared to $m$ and $k$).

Let us consider the sum of all terms of a fixed combinatorial type \emph{and} support in the
expansion of  $\left(\D^{U(N)}_k\right)^m
 \CG^{U(N)}_{\rho(N)}(u_1,\dots,u_N)$, divide it by $\CG^{U(N)}_{\rho(N)}(u_1,\dots,u_N)$ and send $u_i\to 1$. Observe that we
can set $u_i=1$ for all $i$ outside the support \emph{before} summation and, thus, we can further
use the asymptotic estimate for $T$ and its derivatives. After we take all the derivatives we get
the sum of the form appearing in Lemma \ref{Lemma_Vandermonde_sum}. Note that in each derivation of
the exponent in \eqref{eq_x33} a multiple of $N$ pops out. Now Lemma \ref{Lemma_Vandermonde_sum}
yields that as $u_i\to 1$ the sum asymptotically behaves as $N^{\alpha}\cdot C(N)$, where $C(N)$
depends on the combinatorial type, but not on the support. Further $\lim_{N\to\infty} C(N)=C\ne 0$
if $\beta=0$, otherwise $\lim_{N\to\infty} C(N)/N=0$.

Further, we want to sum over all possible supports. Since each support contributes the same terms,
this boils down to the multiplication by $N\choose S$, where $S$ is the number of elements in the
support. We conclude that the total sum of all terms with given combinatorial type is
asymptotically
\begin{equation}
\label{eq_x7}
 N^{S+\alpha} C'(N),\quad
\end{equation}
Further $\lim_{N\to\infty} C'(N)=C'\ne 0$ if $\beta=0$, otherwise $\lim_{N\to\infty} C'(N)/N=0$.

We need to understand for which combinatorial types the power in \eqref{eq_x7} is maximal. Clearly,
we should maximize $\alpha$ and number of elements in the support. First observe that all the
elements of $\{{g_1},\dots g_{\gamma}\}$ should  be either in $\{b_1,\dots,b_\beta\}$, or
$\{a_1,\dots,a_\alpha\}$, or one of the projections of $P$ --- therefore, these elements are
irrelevant in our count.

Further note that $\alpha+\beta+\tau\le mk$, and also the number of the elements in the union of
the sets $\{b_1,\dots,b_\beta\}$, $\{a_1,\dots,a_\alpha\}$ is at most $m$. On the other hand, the
number of elements in the support might be bigger because of the projections of $P$, namely, each
of the $\tau$ elements $t_i$ could have produced a pair in $P$ which would increase the support by
$1$. Now we conclude that if $\beta>0$, then $S+\alpha\le m(k+1)-1$ and $N^{S+\alpha}
C'(N)=o(N^{m(k+1)})$. Thus, the leading terms are those when $\alpha+\tau=mk$ and
$S+\alpha=m(k+1)$. In particular, this means that the function $T_N$ is irrelevant for the leading
term and we can even replace it by $1$: Indeed, since $T_N$ is symmetric, it is irrelevant for it
whether we first symmetrize, then set $u_i=1$ or in the opposite order.

Next, we apply the above analysis for $m=1$ and $m=2$ cases. First, set $m=1$. Then the above
arguments show that the leading term of the asymptotics is the same as the leading term of
$$
 \frac{1}{\prod_{i<j}(u_i-u_j)}\sum_{i=1}^N u_i^k \frac{\partial^k}{\partial
 u_i^k}\left[\exp(N(Q(u_1)+\dots+Q(u_N)))\right] \prod_{i<j}(u_i-u_j)\Bigr|_{u_i=1,\, i=1,\dots,N}.
$$
And the latter has the same leading term as
$$
 \sum_{\ell=0}^k \sum_{i=1}^N N^{k-\ell} {k\choose \ell} u_i^k \frac{\frac{\partial^\ell}{\partial
 u_i^\ell} {\prod_{i<j}(u_i-u_j)}} {\prod_{i<j}(u_i-u_j)} Q'(u_i)^{k-\ell}\Biggr|_{u_i=1,\, i=1,\dots,N}.
$$
Replacing the summation over all supports by one prescribed, we transform the last expression into
\begin{equation}
\label{eq_x8}
 \sum_{\ell=0}^k  \frac{N^{k-\ell} N(N-1)\cdots (N-\ell)}{\ell+1} {k\choose \ell }
\sum_{i=1}^{\ell+1}
  \frac{u_i^kQ'(u_i)^{k-\ell}}{\prod_{j\ne i} (u_i-u_j)}\Biggr|_{u_i=1,\, i=1,\dots,N}.
\end{equation}
Applying Lemma \ref{Lemma_Vandermonde_sum} we conclude that the asymptotics of \eqref{eq_x8} is
\begin{equation}
\label{eq_x42}
 N^{k+1}\sum_{\ell=0}^k \frac{k!}{(k-\ell)!\ell!(\ell+1)!}
 \frac{\partial^{\ell}}{\partial u^{\ell}} (u^k Q'(u)^{k-\ell})\Bigr|_{u=1}.
\end{equation}
Dividing by $N^{k+1}$ we conclude that the limit expectation of the sequence of random measures
$m[\rho(N)]$ has the same moments as prescribed by Theorem \ref{theorem_moment_convergence}.

It remains to prove that the moments of random measures $m[\rho(N)]$, indeed, concentrate and
become deterministic as $N\to\infty$. This would follow from
\begin{equation}
\label{eq_x9}
 \lim_{N\to\infty} \mathbb E\left( \left(\int x^k m[\rho(N)](dx)\right)^2\right)= \lim_{N\to\infty} \left(\mathbb E \left(\int x^k
 m[\rho(N)](dx)\right)\right)^2.
\end{equation}
We already know that the right--hand side of \eqref{eq_x9} is the limit of \
\begin{equation}
\label{eq_x10} \left( \sum_{\ell=0}^k \sum_{i=1}^N N^{-\ell-1} {k\choose \ell}
\sum_{M\subset\{1,\dots,N\}: |M|=\ell, i\notin M} \frac{1}{\prod_{j\in M}(u_i-u_j)} u_i^k
Q'(u_i)^{k-\ell}\right)^2\Biggr|_{u_i=1,\, i=1,\dots,N}.
\end{equation}
 On the other hand, using out asymptotic analysis with $m=2$,
 we get a very similar expression for the leading term of the left--hand side of
\eqref{eq_x9}, i.e.\
\begin{multline}
\label{eq_x11}
 \sum_{\ell=0}^k \sum_{\ell'=0}^k \sum_{i=1}^N \sum_{i'=1}^N N^{-\ell'-\ell-2} {k\choose
\ell} {k\choose \ell'} \sum_{M,M'\subset\{1,\dots,N\}: |M|=\ell, i\notin M, |M'|=\ell', i'\notin
M', M\cap M'=\emptyset}\\ \frac{1}{\prod_{j\in M}(u_i-u_j)} \frac{1}{\prod_{j'\in
M'}(u_{i'}-u_{j'})} u_i^k Q'(u_i)^{k-\ell} u_{i'}^k Q'(u_{i'})^{k-\ell} \Biggr|_{u_i=1,\,
i=1,\dots,N}.
\end{multline}
In fact, the only difference between \eqref{eq_x10} and \eqref{eq_x11} is the condition
$|M|\cap|M'|=\emptyset$ in the second one. However, we know that to understand the leading
asymptotics we need to take only the terms whose support is maximal both in \eqref{eq_x10} and
\eqref{eq_x11} and in such terms the condition $|M|\cap|M'|=\emptyset$ will be satisfied
automatically in both sums.

This finishes the proof of Theorem \ref{theorem_moment_convergence}.

\subsection{Proof of Theorem \ref{theorem_moment_convergence_BCD}} The general plan of the proof here is
the same as for Theorem \ref{theorem_moment_convergence}, i.e.\ we want to apply Proposition
\ref{proposition_moments_and_operators_BCD} and then expand the result into a sum using Leibnitz
rule. The key difference is that we would like to make a change of variables
$z_i=\frac{u_i+u_i^{-1}}{2}$ in operators $\D^{G(N)}_k$ for $G(N)\ne U(N)$. In order to do this, we
note the following identity for $G(N)=SO(2N)$

\begin{multline}
\label{eq_x15_5} \left(u \frac{\partial}{\partial u}\right)^2 f\left(\frac{u+u^{-1}}{2}\right)= u
\frac{\partial}{\partial u} \left( \frac{u-u^{-1}}{2}\right)
f'\left(\frac{u+u^{-1}}{2}\right)\\=\frac{u+u^{-1}}{2}
f'\left(\frac{u+u^{-1}}{2}\right)+\left(\left(\frac{u+u^{-1}}{2}\right)^2-1\right)
f''\left(\frac{u+u^{-1}}{2}\right)\\=\left(z\frac{\partial}{\partial z}+
(z^2-1)\frac{\partial^2}{\partial z^2}\right) f(z) \bigr|_{z=\frac{u+u^{-1}}2}.
\end{multline}
And for $\alpha\in\mathbb R$,
\begin{multline}
\label{eq_x16} \frac{1}{u^{\alpha}-u^{-\alpha}}\left(u \frac{\partial}{\partial u}\right)^2
\left((u^{\alpha}-u^{-\alpha}) f\left(\frac{u+u^{-1}}{2}\right)\right)\\=
\frac{u}{u^\alpha-u^{-\alpha}} \frac{\partial}{\partial u}
\left(\alpha(u^\alpha+u^{-\alpha})f\left(\frac{u+u^{-1}}{2}\right)+
\left(\frac{u-u^{-1}}{2}(u^\alpha-u^{-\alpha}) \right) f'\left(\frac{u+u^{-1}}{2}\right)\right)
 \\=\alpha^2 f\left(\frac{u+u^{-1}}{2}\right)+
 \left(\alpha\frac{(u-u^{-1})(u^\alpha+u^{-\alpha})}{(u^\alpha-u^{-\alpha})}+\frac{u+u^{-1}}{2}\right)
f'\left(\frac{u+u^{-1}}{2}\right)\\+\left(\left(\frac{u+u^{-1}}{2}\right)^2-1\right)
f''\left(\frac{u+u^{-1}}{2}\right).
\end{multline}
Note that when $\alpha=1$ or $\alpha=1/2$, which are the cases we need for $G(N)=Sp(2N)$ and
$G(N)=SO(2N+1)$, respectively, the term
$$
 \alpha\frac{(u-u^{-1})(u^\alpha+u^{-\alpha})}{(u^\alpha-u^{-\alpha})}
$$
becomes a function of $u+u^{-1}$. Therefore, \eqref{eq_x16} transforms into
$$
\left(\alpha^2+ c_\alpha(z)\frac{\partial}{\partial z}+ (z^2-1)\frac{\partial^2}{\partial
z^2}\right) f(z) \bigr|_{z=\frac{u+u^{-1}}2}.
$$

Now we can use exactly the same argument as in the proof of Theorem
\ref{theorem_moment_convergence}, but in variables $z=\frac{u+u^{-1}}{2}$. Note that in the proof
of Theorem \ref{theorem_moment_convergence} we saw that each derivation brings another factor of
$N$ to the asymptotics, thus, for the leading asymptotics only ``maximal'' number of derivatives is
relevant. In other words, in formulas \eqref{eq_x15_5}, \eqref{eq_x16} only the terms with second
derivatives matter, i.e.\ Proposition \ref{proposition_moments_and_operators_BCD} yields that for
$G\ne U(N)$
\begin{multline} \label{eq_x39}
 \mathbb E \left( \int_{\mathbb R} x^{2k} \widehat m^G[\rho(N)](dx) \right)^m\\ \sim
\frac{1}{2^{2mk} N^{m(2k+1)}} \left(\sum_{i=1}^N \left((z_i^2-1)\frac{\partial^2}{\partial
z_i^2}\right)^k \right)^m
 \left(\CG^{G(N)}_{\rho(N)}(u_1,\dots,u_N)\Bigl|_{z_i=\frac{u+u^{-1}}{2}}\right)
 \Biggr|_{z_i=1,i=1,\dots,N}.
\end{multline}
Here and below by $A\sim
 B$ we mean $\lim_{N\to\infty} A/B=1$.

Next, note that characters are polynomials in $u_i+u_i^{-1}$ for systems $B,C$ and for system $D$
when at least one of the arguments is $1$ (and, as we showed in the proof of Theorem
\ref{theorem_moment_convergence}, we need only the values at such points). Therefore, in
\eqref{eq_x39} we can replace
$$
 \left(\CG^{G(N)}_{\rho(N)}(u_1,\dots,u_N)\Bigl|_{z_i=\frac{u+u^{-1}}{2}}\right)
$$
by an analytic function $\widehat \CG^{G(N)}_{\rho(N)}(z_1,\dots,z_N)$ such that
$$
 \lim_{N\to\infty} \frac{1}{N} \ln( \widehat \CG^{G(N)}_{\rho(N)}(z_1,\dots,z_k, 1^{N-k}) ) = \widehat Q(z_1)+\dots+\widehat Q(z_k),
$$
 and
$$
 \widehat Q\left(\frac{u+u^{-1}}2\right)= Q(u).
$$

Now the argument of Theorem \ref{theorem_moment_convergence} leading to the concentration for the
moments and formula for the limit \eqref{eq_x42} can be repeated. This yields the concentration for
the moments in series $B,C,D$ and the following formula for the moments of the limit measure:
\begin{equation}
\label{eq_limit_moments_BCD_2}
 \int_{\mathbb R} x^{2k} \mes^G(dx)= 2^{-2k} \sum_{\ell=0}^{2k}
 \frac{(2k)!}{\ell! (\ell+1)! (2k-\ell)!} \frac{\partial^\ell}{\partial z^{\ell}}\left((z^2-1)^k
 \widehat Q'(z)^{2k-l}\right)\Biggr|_{z=1}.
\end{equation}

\subsection{Proof of Theorem \ref{theorem_moment_link_to_PP} and Theorem \ref{theorem_intertwining}.}

Theorem \ref{theorem_moment_link_to_PP} is an immediate corollary of the following result, which,
in turn, follows from the results of Section 4 of \cite{PP} and \cite{Popov}, \cite{Popov2}.

\begin{proposition}[\cite{PP}, \cite{Popov}, \cite{Popov2}]
 For each $G=A,B,C,D$ and each $k=1,2,\dots$ there exist $\ell(k,G)$ multivariate polynomials $P_1$,\dots, $P_{\ell(k,G)}$
 and $\ell(k,G)$ functions $f_1(N),\cdots, f_{\ell(k,G)}(N)$ such that for any $\lambda\in\hG(N)$
 with the notations
$$
 s_k[\lambda]=\int_{\mathbb R} x^k m^G[\lambda](dx),\quad c_k[\lambda]=\int_{\mathbb R} x^k m^G_{PP}[\lambda](dx), \quad k=0,1,2,\dots
$$
we have
\begin{equation}
\label{eq_x21}
 c_k[\lambda]=\sum_{i=1}^{\ell(k,G)} f_i(N) P_i(s_1[\lambda],s_2[\lambda],\dots,s_k[\lambda]).
\end{equation}
The functions $f_i(N)$ have limits as $N\to\infty$ such that \eqref{eq_x21} asymptotically turns
into
\begin{equation}
\label{eq_x22} 1 - \sum_{k=0}^{\infty} c_k z^{k+1} = \exp \left(- \sum_{k=0}^{\infty} s_k z^{k+1}
\right).
\end{equation}
\end{proposition}
In fact, \cite[Theorem 3]{PP} and \cite[Eq.\ 2.11--2.12]{Popov2} contain explicit formulas for the
polynomials $P_i$ and functions $f_i$ from which the limit transition to \eqref{eq_x22} is
immediate.

\bigskip

Let us now turn to  Theorem \ref{theorem_intertwining}. Observe, that the existence of $Q(\rho)$
can be deduced from Theorem \ref{theorem_moment_link_to_PP}. Indeed, any probability measure
$\rho$ with compact support and bounded by $1$ density can be approximated as $N\to\infty$ limit
of $m^A[\lambda(N)]$ with a suitable $\lambda(N)\in\widehat U(N)$, then the limit of
$m^A_{PP}[\lambda(N)]$ gives $Q(\rho)$. The intertwining property of Theorem
\ref{theorem_intertwining} is a simple corollary of Theorem \ref{Theorem_linearization}, which we
will prove below.

\section{Concentration phenomena}

In this section we prove the main results announced in Section \ref{Section_results}.

\subsection{Consistency check}
First, we would like to check that Theorem \ref{Theorem_character_asymptotics} and Theorem
\ref{theorem_moment_convergence} agree with each other.
\begin{lemma}
\label{Lemma_consistency} Let $\mes$ be a probability measure on $\mathbb R$ with compact support,
then its moments $M_k(\mes)$ can be computed through
\begin{equation}
\label{eq_x12}
 M_k(\mes)= \sum_{\ell=0}^k \frac{k!}{\ell! (\ell+1)! (k-\ell)!} \frac{\partial^\ell}{\partial u^k}\left(u^k
 \left(\frac{\partial\H_\mes(u)}{\partial u}\right)^{k-l}\right)\Biggr|_{u=1}.
\end{equation}
\end{lemma}
\noindent {\bf Remark.} For measures which can be obtained as weak limits of $m[\lambda(N)]$ with
regular sequence $\lambda(N)$ this is an immediate combination of Theorem
\ref{Theorem_character_asymptotics} and Theorem \ref{theorem_moment_convergence}. However, only
measures with density with respect to the Lebesgue measure at most $1$ can be obtained in such a
way.
\begin{proof}[Proof of Lemma \ref{Lemma_consistency}]
Using integral representation for the derivative (i.e.\ Cauchy formula) \eqref{eq_x12} can be
transformed into
\begin{multline}
\label{eq_moment_integral_formula}
 M_k(\mes)=\frac{1}{2\pi\ii} \oint_{1}  \sum_{l=0}^k  \frac{1}{l+1} {k\choose l} \frac{z^k
 H_\mes'(z)^{k-l}}{(z-1)^{l+1}} dz
 \\= \frac{1}{2\pi\ii} \oint_{1} \frac{z^k H_\mes'(z)^{k+1}}{k+1} \sum_{l=-1}^k  {{k+1} \choose {l+1}}
 \frac{1}{
 H_\mes'(z)^{l+1}(z-1)^{l+1}} dz\\= \frac{1}{2\pi\ii} \oint_{1}  \frac{z^k H_\mes'(z)^{k+1}}{k+1}
 \left(1+\frac{1}{H_\mes'(z)(z-1)}\right)^{k+1} dz\\= \frac{1}{2(k+1)\pi\ii} \oint_{1} \frac{dz}{z}
 \left(zH_\mes'(z)+\frac{z}{z-1}\right)^{k+1},
\end{multline}
where the integration goes over a small positively oriented contour around $1$. On the other hand,
the definition of $\H_\mes(z)$ yields
$$
 \frac{x}{xe^x H_\mes'(e^x) + \frac{xe^x}{e^x-1}} = \frac{1}{R_\mes(x)+\frac{1}{x}}= (S_\mes(z))^{(-1)},
$$
where $(\cdot)^{(-1)}$ is the functional inversion and $S_\mes(z)$ is the moment generating
function:
$$
 S_\mes(z)=z+M_1(\mes) z^2 + M_2(\mes) z^3+\dots.
$$
Now using Lagrange inversion theorem (in the form of Lagrange--B\"{u}rmann formula, see e.g.\
\cite[Section 5.4]{Stanley}), we get
\begin{multline}
\label{eq_moment_integral_formula_2}
 M_k(\mes) = [z^{k+1}](S(z))=\frac{1}{k+1} [w^k] \left(we^w H_\mes'(e^w) + \frac{we^w}{e^w-1}\right)^{k+1}
 \\ = \frac{1}{2(k+1)\pi\ii} \oint_0 \left(e^w H_\mes'(e^w)+\frac{e^w}{e^{w}-1}\right)^{k+1} dw,
\end{multline}
where the integration goes over a small contour around $0$. It remains to observe that the change
of variables $w=\ln(z)$ transforms \eqref{eq_moment_integral_formula_2} into
\eqref{eq_moment_integral_formula}.
\end{proof}

For other root systems we need the following statement:

\begin{lemma}
\label{Lemma_consistency_BCD} Let $\mes$ be a probability measure on $\mathbb R$ symmetric with
respect to the origin and with compact support. Then its even moments $M_{2k}(\mes)$ can be
computed through
\begin{equation}
\label{eq_x32}
 M_{2k}(\mes)= 2^{-2k} \sum_{\ell=0}^{2k}
 \frac{(2k)!}{\ell! (\ell+1)! (2k-\ell)!} \frac{\partial^\ell}{\partial z^{\ell}}\left((z^2-1)^k
 \widehat H'(z)^{2k-l}\right)\Biggr|_{z=1}.
\end{equation}
where
\begin{equation}
\label{eq_x31}
 \widehat H\left(\frac{x+x^{-1}}2\right)=2H_\mes(x)+\ln(x).
\end{equation}
\end{lemma}
\begin{proof}
By the same argument as in Lemma \ref{Lemma_consistency} we transform \eqref{eq_x32} into
\begin{equation}
\label{eq_x17}
 M_{2k}(\mes)= \frac{2^{-2k}}{2(2k+1)\pi\ii} \oint_{1}  (z^2-1)^k
 \left(\widehat H'(z)+\frac{1}{z-1}\right)^{2k+1}dz.
\end{equation}
Thus, changing the variables $z=(x+x^{-1})/2$ in \eqref{eq_x17}, we get (additional factor of $2$
appears because the contour is doubled)
\begin{multline}
\label{eq_x18}
 M_{2k}(\mes)\\= \frac{2^{-2k}}{4(2k+1)\pi\ii} \oint_{1}  \left(\frac{x-x^{-1}}{2}\right)^{2k+1}
 \left(\frac{4x}{x-x^{-1}}(H_{\mes})'(x)+\frac{2}{x-x^{-1}}+\frac{2}{x+x^{-1}-2}\right)^{2k+1}\frac{dx}{x}
\\=\frac{2^{-2k}}{2(2k+1)\pi\ii} \oint_{1}
 \left(2x(H_{\mes})'(x)+\frac{2x}{x-1}\right)^{2k+1}\frac{dx}{x}
\end{multline}

Further, $\H_{\mes}(z)$ satisfies
$$
 \frac{x}{xe^x (\H_{\mes})'(e^x) + \frac{xe^x}{e^x-1}} = \frac{1}{R_{\mes}(x)+\frac{1}{x}}= (S_{\mes}(z))^{(-1)},
$$
where $(\cdot)^{(-1)}$ is the functional inversion and $S_\mes(z)$ is the moment generating
function:
$$
 S_\mes(z)=z+M_1(\mes) z^2 + M_2(\mes) z^3+\dots.
$$
Therefore,
\begin{multline}
\label{eq_x19}
 M_{2k}(\mes) = [z^{2k+1}](S(z))=\frac{1}{2k+1} [w^{2k}] \left(we^w (\H_{\mes})'(e^w) + \frac{we^w}{e^w-1}\right)^{2k+1}
 \\ = \frac{1}{(2k+1)\pi\ii} \oint_0 \left(e^w (\H_{\mes})'(e^w)+\frac{e^w}{e^{w}-1}\right)^{2k+1} dw,
\end{multline}
change of variables $x=e^w$ transforms \eqref{eq_x18} into \eqref{eq_x19}.
\end{proof}

\subsection{Proof of Theorems \ref{Theorem_main}, \ref{Theorem_main_PP}, \ref{Theorem_linearization}.}

Note that our definitions and the fact that the character of a tensor product is the product of the
characters of factors, imply that the character--generating function of measure
$\rho^{\lambda^1(N)\otimes\cdots\otimes\lambda^k(N)}$ is
\begin{equation}
\label{eq_x43}
 \CG^{G(N)}_{\rho^{\lambda^1(N)\otimes\cdots\otimes\lambda^k(N)}}(u_1,\dots,u_N)=\prod_{i=1}^k
 \frac{\chi^{G(N)}_{\lambda^i(N)}(u_1,\dots,u_N)}{\chi^{G(N)}_{\lambda^i(N)}(1^N)}.
\end{equation}
 Similarly, the character generating function for $\rho^{\alpha,\lambda^1(N)}$ is
\begin{equation}
\label{eq_x44}
 \CG^{G(N)}_{\rho^{\alpha, \lambda^1(N)}}(u_1,\dots,u_{\lfloor \alpha N\rfloor})=
 \frac{\chi^{G(N)}_{\lambda^1(N)}\left(u_1,\dots,u_{\lfloor \alpha N\rfloor}, 1^{N-\lfloor \alpha N\rfloor}\right)}{\chi^{G(N)}_{\lambda^1(N)}(1^N)}.
\end{equation}
Now we can apply Theorem \ref{Theorem_character_asymptotics} together with Theorems
\ref{theorem_moment_convergence} and \ref{theorem_moment_convergence_BCD}. They yield that the
moments of random measures $m^{U(N)}[\rho^{\lambda^1(N)\otimes\cdots\otimes\lambda^k(N)}]$,
$m^{U(N)}[\rho^{\alpha,\lambda^1(N)}]$ and $\widehat
m^{G(N)}[\rho^{\lambda^1(N)\otimes\cdots\otimes\lambda^k(N)}]$, $\widehat
m^{G(N)}[\rho^{\alpha,\lambda^1(N)}]$
 converge (in probability) to deterministic numbers. Since we are dealing with measures with compact support here (as follows from the
Littlewood--Richardson rule, the measures corresponding to tensor products in our settings have a
finite support), the moments uniquely define the corresponding measures. Thus, the above random
measures converge in the sense of moments (in probability) to deterministic ones.   Further, the
measures $\widehat m^{G(N)}[\rho^{\lambda^1(N)\otimes\cdots\otimes\lambda^k(N)}]$, $\widehat
m^{G(N)}[\rho^{\alpha,\lambda^1(N)}]$ and  $
m^{G(N)}[\rho^{\lambda^1(N)\otimes\cdots\otimes\lambda^k(N)}]$, $
m^{G(N)}[\rho^{\alpha,\lambda^1(N)}]$ asymptotically differ by the shift by $1/2$. Therefore, the
latter measures also converge in the sense of moments. This proves Theorem \ref{Theorem_main}. Now
applying Theorem \ref{theorem_moment_link_to_PP} we conclude that the Perelomov--Popov measures
also converge, which proves Theorem \ref{Theorem_main_PP}.

Moreover, for $G(N)=U(N)$ \eqref{eq_x43}, \eqref{eq_x44} yields for every $m=1,2,\dots$
\begin{multline*}
\lim_{N\to\infty}\frac{1}{N}
 \ln\left(\CG^{G(N)}_{\rho^{\lambda^1(N)\otimes\cdots\otimes\lambda^k(N)}}(u_1,\dots,u_m, 1^{N-m})\right)\\=\sum_{i=1}^k\lim_{N\to\infty}
 \frac{1}{N}\ln\left(\frac{\chi^{G(N)}_{\lambda^i(N)}(u_1,\dots,u_m,1^{N-m})}{\chi^{G(N)}_{\lambda^i(N)}(1^N)}\right),
\end{multline*}
\begin{multline*}
\lim_{N\to\infty}\frac{1}{\lfloor \alpha N\rfloor}
 \ln\left(\CG^{G(\lfloor \alpha N\rfloor)}_{\rho^{\alpha, \lambda^1(N)}}(u_1,\dots,u_m, 1^{\lfloor \alpha N\rfloor-m})\right)\\=
\frac{1}{\alpha} \lim_{N\to\infty}\frac{1}{N}
 \ln\left(\CG^{G(N)}_{\rho^{\lambda^1(N)}}(u_1,\dots,u_m, 1^{N-m})\right).
\end{multline*}
Hence, comparing \eqref{eq_x12} with \eqref{eq_limit_moments_A} and \eqref{eq_limit_of_logarithm},
and noting that function $\H^A_{\mes}(u)$ uniquely defines the moments of $\mes$ and, hence, the
measure $\mes$, we conclude that
 \begin{equation}
  \H_{\mes^1\otimes\dots\otimes\mes^k}(u)=\H_{\mes^1}(u)+\dots+\H_{\mes^k}(u).
 \label{eq_linear_1_A}
 \end{equation}
and
 \begin{equation}
 \H_{pr_\alpha^\otimes(\mes)}(u)=\frac{1}{\alpha}\H_{\mes}(u)
 \label{eq_linear_2_A}.
 \end{equation}
For $G(N)\ne U(N)$ we similarly compare \eqref{eq_x32} with \eqref{eq_limit_moments_BCD} and
\eqref{eq_limit_of_logarithm}. Observing that the asymptotic shift of measure by $1/2$ by which
Theorem \ref{Theorem_character_asymptotics} and Theorem \ref{theorem_moment_convergence_BCD}
differ, translates precisely in the additional term $\ln(x)$ in \eqref{eq_x31}, we again conclude
that identities \eqref{eq_linear_1_A}, \eqref{eq_linear_2_A} hold.

 Since
$$
 R^{quant}_{\mes}(z)=  \left( u\frac{\partial \H_{\mes}(u)}{\partial u}\right)\Biggr|_{u=e^z},
$$
the identities \eqref{eq_linear_1_A} and \eqref{eq_linear_2_A} imply \eqref{eq_linear_final_1} and
\eqref{eq_linear_final_2}, respectively.

For the Perelomov--Popov measures, observe that the identity between moment generating functions in
Theorem \ref{theorem_moment_link_to_PP} is equivalent to the following identity between Voiculescu
$R$--transforms
$$
 R_\mes(z)=R_{\mes_{PP}}(1-e^{-z}) + \frac{1}{1-e^{-z}} -\frac{1}{z}.
$$
Thus, \eqref{eq_linear_final_1} and \eqref{eq_linear_final_2} imply that
 \begin{equation}
  R_{\mes^1\boxplus\dots\boxplus\mes^k}(1-e^{-z})  =
  R_{\mes^1}(1-e^{-z}) +\dots+R_{\mes^k}(1-e^{-z}),
 \label{eq_linear_1_PP}
 \end{equation}
and
 \begin{equation}
 R_{pr_\alpha^\boxplus(\mes)}(1-e^{-z})=\frac{1}{\alpha}R_{\mes}(1-e^{-z})
 \label{eq_linear_2_PP}.
 \end{equation}
 Changing the variables $u=1-e^{-z}$ we arrive at \eqref{eq_linear_final_1_PP} and
 \eqref{eq_linear_final_2_PP} which finishes the proof of Theorem \ref{Theorem_linearization}.

\section{Appendix: Asymptotics of characters}
\label{Section_appendix} The aim of this appendix is to provide the necessary details for the
proof of Theorem \ref{Theorem_character_asymptotics}. The proof goes in two steps. The first one
is to study the $k=1$ case and in the second step we reduce general $k$ to $k=1$ using approximate
multiplicativity of the characters found in \cite{GP}.

\begin{proposition}
\label{Proposition_character_asymptotics_k1}
 Suppose that $\lambda(N)\in \widehat U (N)$ is a regular sequence of signatures, such that
 $$
  \lim_{N\to\infty} m^A[\lambda(N)]=\mes.
 $$
 Then we have
 \begin{equation}
 \label{eq_limit_of_logarithm_k1}
  \lim_{N\to\infty} \frac{1}{N}
  \ln\left(\frac{\chi^{U(N)}_{\lambda(N)}(x,1^{N-1})}{\chi^{U(N)}_{\lambda(N)}(1^N)}\right)=
  \H_\mes(x),
 \end{equation}
 where  the convergence is uniform over an open complex neighborhood of $1$.
\end{proposition}
\begin{proof}

This statement for real $x$ and without uniformity estimate first appeared in \cite{GM}, see
Theorem 1.2 and explanation on the relation to Schur polynomials at the end of Section 1.1 there. A
closely related statement, which corresponds to $\beta=1$ (``orthogonal'') matrix integrals, as
opposed to \eqref{eq_limit_of_logarithm_k1} corresponding to $\beta=2$ (``unitary'') matrix
integrals, for complex $x$ is \cite[Theorem 1.4]{GM}. An alternative approach based on contour
integral representation for Schur polynomials was suggested in \cite{GP} and
\eqref{eq_limit_of_logarithm_k1} (again, without uniformity estimate) is a corollary of the results
of \cite{GP}, see Proposition 4.1 and remark at the end of Section 4.2 there. The uniformity,
probably, can be proved both by methods of \cite{GM} or \cite{GP}. Let us fill in the required
details for the latter approach. We should check that remainders in \cite[Proposition 4.1]{GP} are
\emph{uniformly} small.

First, we should bound $\ln(\mathcal Q(w,\lambda,f))$ in \cite[Lemma 4.4]{GP}. This logarithm is
\begin{equation}
\label{eq_difference}
 \left(\sum_{j=1}^N \ln\left(w-\frac{\lambda_j(N)+N-j}{N}\right)\right)-N\int_0^1\ln(w-f(t)-1+t) dt,
\end{equation}
where $w$ is a large enough complex number (the exact set where $w$ varies depends on $x$).
Clearly, when $\lambda(N)$ is regular, the bound of \cite{GP} claiming that \eqref{eq_difference}
is $o(N)$ becomes uniform over large complex $w$.

Second, we should make sure that the logarithm of \cite[(4.12)]{GP} divided by $N$ tends to $0$
\emph{uniformly} over large complex $w_0$. But again this immediately follows from the definitions
of $\tilde \delta$, and $u$ there.
\end{proof}

An analogue of Proposition \ref{Proposition_character_asymptotics_k1} for the root systems $B$, $C$
and $D$  is obtained through Propositions
\ref{Proposition_Schur_Symplectic_1}--\ref{Proposition_Schur_Ortho_2} which \emph{reduce}
normalized characters of symplectic and orthogonal groups to the normalized characters of unitary
groups. We still do not know any conceptual representation--theoretic explanation for these
reductions and it would be very interesting to find one.

Recall that characters of unitary group $U(N)$ are identified with \emph{Schur polynomials}
$s_\lambda(u_1,\dots,u_N)$ and the following integral representation for them was found in
\cite[Theorem 1.1]{GP}:
\begin{equation}
\label{eq_Schur_integral}
 \frac{s_\lambda(x,1^{N-1})}{s_\lambda(1^N)}=\frac{(N-1)!}{2\pi\ii(x-1)^{N-1}}\oint
 \frac{x^z dz}{\prod_{i=1}^N(z-(\lambda_i+N-i))}.
\end{equation}

\begin{proposition}
\label{Proposition_Schur_Symplectic_1} For any signature $\lambda\in\widehat{Sp}(2N)$  we have
\begin{equation}
\frac{\chi^{Sp(2N)}_{\lambda}(x,1^{N-1})}{\chi^{Sp(2N)}_{\lambda}(1^{N})} =
\frac{2}{x+1}\frac{s_\nu(x,1^{2N-1})}{s_\nu(1^{2N})} ,
\end{equation}
where $\nu\in\widehat U (2N)$ is $(\lambda_1+1,\dots,\lambda_N+1,-\lambda_N,\dots,-\lambda_1)$.
\end{proposition}
\begin{proof} This is \cite[Proposition 3.19]{GP}. \end{proof}

\begin{proposition}
\label{Proposition_Schur_Ortho_1} For any signature $\lambda\in\widehat{SO}(2N+1)$  we have
\begin{equation}
\frac{\chi^{SO(2N+1)}_{\lambda}(x,1^{N-1})}{\chi^{SO(2N+1)}_{\lambda}(1^{N})} =
\frac{s_\nu(x,1^{2N-1})}{s_\nu(1^{2N})} ,
\end{equation}
where $\nu\in\widehat U (2N)$ is $(\lambda_1,\dots,\lambda_N,-\lambda_N,\dots,-\lambda_1)$.
\end{proposition}
\begin{proof}
Following the method of \cite[Section 3]{GP} one proves the following integral formula for
characters of $SO(2N+1)$, which can be also found in \cite[Section 2.1]{HJ}.
\begin{equation}
\label{eq_x14} \frac{\chi^{SO(2N+1)}_{\lambda}(x,1^{N-1})}{\chi^{SO(2N+1)}_{\lambda}(1^{N})}=
\frac{(2N-1)!}{2\pi\ii ( (x^{1/2}-x^{-1/2}))^{2N-1}} \oint \frac{(x^{z}-x^{-z}) dz} {\prod_{i=1}^n
(z^2-(\lambda_i+N-i+1/2)^2)},
\end{equation}
where the integration goes around the poles at $\lambda_i+N-i+1/2$, $i=1,\dots,N$. We claim that
\eqref{eq_x14} is the same as
\begin{equation}
\label{eq_x15} \frac{(2N-1)!}{2\pi\ii ( (x^{1/2}-x^{-1/2}))^{2N-1}} \oint \frac{(x^{z}) dz}
{\prod_{i=1}^n (z-(\lambda_i+N-i+1/2))(z+(\lambda_i+N-i+1/2))},
\end{equation}
with integration going around \emph{all} poles of the integrand. Indeed, to prove this just expand
both \eqref{eq_x14} and \eqref{eq_x15} as sums of residues. Further, shifting $z$ by $N-1/2$, we
arrive at
$$
\frac{(2N-1)!}{2\pi\ii ( (x-1))^{2N-1}} \oint \frac{(x^{z}) dz} {\prod\limits_{i=1}^N
\bigl(z-(\lambda_i+2N-i)\bigr) \prod\limits_{i=N+1}^{2N}\bigl(z-(-\lambda_{2N+1-i}+(2N-i))\bigr)},
$$
which is the integral formula \eqref{eq_Schur_integral} for
$\frac{s_\nu(x,1^{2N-1})}{s_\nu(1^{2N})}$.
\end{proof}

\begin{proposition}
\label{Proposition_Schur_Ortho_2} Take $N>1$ and a signature $\lambda\in\widehat{SO}(2N)$. If
$\lambda_N=0$, then we have
\begin{equation}
\label{eq_x36} \frac{\chi^{SO(2N)}_{\lambda}(x,1^{N-1})}{\chi^{SO(2N)}_{\lambda}(1^{N})} =
\frac{s_\nu(x,1^{2N-2})}{s_\nu(1^{2N-1})} ,
\end{equation}
where $\nu\in\widehat U (2N-1)$ is
$(\lambda_1,\dots,\lambda_{N-1},0,-\lambda_{N-1},\dots,-\lambda_1)$. If $\lambda_N\ne 0$,
\begin{equation}
\label{eq_x37} \frac{\chi^{SO(2N)}_{\lambda}(x,1^{N-1})}{\chi^{SO(2N)}_{\lambda}(1^{N})} =
 \left(1+(1-x^{-1}) \frac{x\frac{\partial}{\partial x}-N}{2N-1}\right)
 \frac{s_\nu(x,1^{2N-1})}{s_\nu(1^{2N})} ,
\end{equation}
where $\nu\in\widehat U (2N)$ is
$(\lambda_1,\dots,\lambda_{N-1},|\lambda_N|,1-|\lambda_N|,1-\lambda_{N-1},\dots,1-\lambda_1)$.
\end{proposition}
\begin{proof}
First, note that in the Weyl formula \eqref{eq_character_D} for the character of $SO(2N)$, the
second term vanishes when at least one of $u_i$ is $1$, which is our case. The first term does not
change under the transformation $\lambda_N\to-\lambda_N$ and, thus, we assume $\lambda_N\ge 0$.
Following the approach of \cite[Section 3]{GP} we write
\begin{multline}
\label{eq_x34}
 \frac{\det\left[u_i^{\lambda_j+N-j}+u_i^{-(\lambda_j+N-j)}\right]_{i,j=1}^N}{\prod\limits_{i<j}(u_i+u_i^{-1}-(u_j+u_j^{-1}))}
\\=\sum_{j=1}^N (-1)^{j-1}\frac{
u_1^{\lambda_j+N-j}+u_1^{-(\lambda_j+N-j)}}{\prod_{i=2}^N(u_1+u_1^{-1}-(u_j+u_j^{-1}))}\cdot
\frac{\det\left[M^{(i)}\right]}{\prod\limits_{1<\alpha<\beta\le
N}(u_\alpha+u_\alpha^{-1}-(u_\beta+u_\beta^{-1}))},
\end{multline}
where $M^{(i)}$ is the submatrix of $\left[u_i^{\lambda_j+N-j}+u_i^{-(\lambda_j+N-j)}\right]$
obtained by crossing out the $i$th row and column. Now we substitute $u_1=x$, $u_2=u_3=\dots=u_N=1$
in the right side \eqref{eq_x34} and divide by the result of the substitution
$u_1=u_2=\dots=u_N=1$. We get (essentially we are using Weyl's dimension formula for the dimension
of representation of $SO(2N)$):
\begin{equation}
\label{eq_x35}
 \frac{\chi^{SO(2N)}_{\lambda}(x,1^{N-1})}{\chi^{SO(2N)}_{\lambda}(1^{N})}=
 \frac{\prod_{i=0}^{N-2}((N-1)^2-i^2)}{\prod_{i=2}^N(x+x^{-1}-2)} \sum_{j=1}^N
 \dfrac{x^{\lambda_j+N-j}+x^{-(\lambda_j+N-j)}}{\prod_{i\ne j}((\lambda_j+N-j)^2-(\lambda_i+N-i)^2)}
\end{equation}
 When $\lambda_N\ne 0$, we transform the last sum into a contour integral
$$
\frac{\prod_{i=0}^{N-2}((N-1)^2-i^2)}{2\pi\ii (x+x^{-1}-2)^{N-1}} \oint \frac{2z(x^z+x^{-z})}{
\prod_{i=1}^N(z^2-(\lambda_i+N-i)^2)}dz,
$$
with the integration contour enclosing the singularities at $\lambda_i+N-i$, $i=1,\dots,N$.
Equivalently,
$$
\frac{(2N-2)!}{2\pi\ii(x+x^{-1}-2)^{N-1}} \oint \frac{z x^z}{
\prod_{i=1}^N(z^2-(\lambda_i+N-i)^2)}dz,
$$
with the integration contour enclosing \emph{all} singularities. Shifting $z$ by $N$ we arrive at
$$
\frac{(2N-1)! (1-x^{-1}) }{2\pi\ii(2N-1)(x-1)^{2N-1}} \oint \frac{(z-N) x^{z} dz}{
\prod_{i=1}^N(z-(\lambda_i+2N-i))\prod_{i={N+1}}^{2N}(z-(-\lambda_{2N+1-i}+1+2N-i))},
$$
Using \eqref{eq_Schur_integral} the last expression is readily identified with
$$
 \frac{(1-x^{-1})}{2N-1} \left( (x-1)^{1-2N}\circ (x \frac{\partial}{\partial x})\circ(x-1)^{2N-1}-N\right)
 \left(\frac{s_\nu(x,1^{2N-1})}{s_\nu(1^{2N})}\right),
$$
 which is \eqref{eq_x37}.
On the other hand, if $\lambda_N=0$, then \eqref{eq_x35} is
$$
\frac{(2N-2)!}{2\pi\ii(x+x^{-1}-2)^{N-1}} \oint \frac{x^z}{z
\prod_{i=1}^{N-1}(z^2-(\lambda_i+N-i)^2)}dz,
$$
with the integration contour enclosing all singularities. Shifting $z$ by $N-1$ we arrive at
$$
\frac{(2N-2)!}{(x-1)^{2N-2}} \oint
\frac{x^zdz}{\left(\prod_{i=1}^{N-1}(z-(\lambda_i+2N-i))\right)(z-N)
\left(\prod_{i=N+1}^{2N-1}(z-(-\lambda_{2N-i}+2N-i))\right)},
$$
which is the integral representation \eqref{eq_Schur_integral} for$
 \frac{s_\nu(x,1^{2N-2})}{s_\nu(1^{2N-1})}.
$
\end{proof}

The above propositions imply the following.
\begin{corollary}
\label{corollary_character_asymptotics_k1}
 Suppose that $\lambda(N)\in \hG (N)$ is a regular sequence of signatures, such that
 $$
  \lim_{N\to\infty} m^G[\lambda(N)]=\mes.
 $$
 Then we have
 \begin{equation}
 \label{eq_limit_of_logarithm_k1_BCD}
  \lim_{N\to\infty} \frac{1}{\hat N}
  \ln\left(\frac{\chi^{G(N)}_{\lambda(N)}(x,1^{N-1})}{\chi^{G(N)}_{\lambda(N)}(1^N)}\right)=
  \H_\mes(x),
 \end{equation}
 where  the convergence is uniform over an open complex neighborhood of $1$. $\hat N=N$ for
 unitary group $U(N)$ and $\hat N=2N$ for orthogonal and symplectic groups.
\end{corollary}
\begin{proof}
 For $G(N)=U(N)$ this is Proposition \ref{Proposition_character_asymptotics_k1}. For $G(N)=Sp(2N)$
 we use Proposition \ref{Proposition_Schur_Symplectic_1}. Since the
 $\frac{1}{N}\ln(\frac{2}{x+1})$ vanishes as $N\to\infty$, the result again follows from
 Proposition \ref{Proposition_character_asymptotics_k1}. Similarly for odd orthogonal group
 $SO(2N+1)$ and $\lambda_N=0$ case for
 even orthogonal group $SO(2N)$ we use Propositions \ref{Proposition_Schur_Ortho_1}, \ref{Proposition_Schur_Ortho_2},
 \ref{Proposition_character_asymptotics_k1}. Finally, for even orthogonal group $SO(2N)$ and
 $\lambda_N\ne 0$, using Proposition \ref{Proposition_Schur_Ortho_2} and
$$
 \frac{s_{\nu(N)}(x,1^{2N-1})}{s_{\nu(N)}(1^{2N})}= e^{2N H_{\mes}(x)} T_N(x),
$$
with
\begin{equation}
\label{eq_x38} \lim_{N\to\infty} \frac{1}{N}\ln(T_N(x))=0,
\end{equation}
 we get
 \begin{multline}
 \label{eq_x49}
\frac{\chi^{G(N)}_{\lambda(N)}(x,1^{N-1})}{\chi^{G(N)}_{\lambda(N)}(1^N)}= \left(1+(1-x^{-1})
\frac{x\frac{\partial}{\partial x}-N}{2N-1}\right)\Bigl(e^{2N H_{\mes}(x)}
T_N(x)\Bigr)\\
=\left(1+(1-x^{-1}) \frac{2Nx H'_{\mes}(x)+ x T'_N(x)/T_N(x)-N}{2N-1}\right)\Bigl(e^{2N
H_{\mes}(x)} T_N(x)\Bigr)
\end{multline}
Note that since \eqref{eq_x38} involves uniformly converging analytic functions, we can
differentiate it, which yields that $\lim_{N\to\infty} \frac{1}{N} {T'_N(x)}{T_N(x)}=0$ and, thus,
the term involving $T'_N(x)$ is negligible in \eqref{eq_x49}. Now taking logarithm of
\eqref{eq_x49} and dividing by $2N$, we get $H_{\mes}(x)$, as desired.
\end{proof}

The reduction of general $k$ to $k=1$ is given in the following statement.

\begin{proposition}
 \label{Proposition_character_asymptotics_k}
 Suppose that $\lambda(N)\in \hG (N)$ is a sequence of signatures such that
 \begin{equation}
  \lim_{N\to\infty} \frac{1}{N}
  \ln\left(\frac{\chi^{G(N)}_{\lambda(N)}(x,1^{N-1})}{\chi^{G(N)}_{\lambda(N)}(1^N)}\right)=
  Q(x),
 \end{equation}
 where  the convergence is uniform over a compact set $M\subset \mathbb C$, then for any $k\ge 1$
 \begin{equation}
  \lim_{N\to\infty} \frac{1}{N}
  \ln\left(\frac{\chi^{G(N)}_{\lambda(N)}(u_1,\dots,u_k,1^{N-k})}{\chi^{G(N)}_{\lambda(N)}(1^N)}\right)=
  Q(u_1)+\dots+Q(u_k),
 \end{equation}
 where  the convergence is uniform over $M^k\subset \mathbb C$.
\end{proposition}
\begin{proof}
 For unitary groups this is \cite[Corollary 3.11]{GP} (see also \cite[Theorem 1.7]{GM} and \cite[Theorem 2]{Colins-Sniady}
 for related
 statements in the case of real $u_i$.)
 For symplectic and orthogonal groups (and even for more general multivariate Jacobi polynomials) this is
 proved in the same way as \cite[Corollary 3.11]{GP} using the formulas of \cite[Theorem 3.17]{GP} and
 \cite[Theorem 3.21]{GP}.
\end{proof}

\end{document}